\documentclass[a4paper,11pt,pdfmx]{article}
\usepackage{amsfonts}
\usepackage{amssymb}
\usepackage{amsmath}
\usepackage{amsthm}
\usepackage{bm}
\usepackage{here}
\usepackage{tikz}
\usetikzlibrary{intersections,calc,arrows}
\usepackage{color}
\usepackage{cases}
\numberwithin{equation}{section}

\theoremstyle{definition}
\newtheorem{defn}{Definition}[section]
\newtheorem{rem}[defn]{Remark}
\theoremstyle{plain}
\newtheorem{lemm}[defn]{Lemma}
\newtheorem{prop}[defn]{Proposition}
\newtheorem{thm}[defn]{Theorem}
\newtheorem{cor}[defn]{Corollary}

\title{Jackson integral representation for {a multiple} $q$-hypergeometric series and an extension of the $q$-Riemann--Papperitz system}
\author{
	Takahiko Nobukawa
	\footnote{{Department of Education, Kogakkan University, Kodakushimoto-cho, Ise 516-8555, Japan.\\
			E-mail: t-nobukawa@kogakkan-u.ac.jp}
		\\\ \\
		keywords:
		$q$-hypergeometric series, Jackson integral, multivariable $q$-difference system, Bailey's formula.}
}
\date{}
\allowdisplaybreaks

\begin{document}
	\maketitle
	\begin{abstract}
		We give {a Jackson integral} representation for Kajihara's $q$-hypergeometric series $W^{M,2}$.
		We construct a $q$-difference system that corresponds to this integral.
		This system is an extension of the variant of $q$-hypergeometric equation of degree three, defined by Hatano--Matsunawa--Sato--Takemura.
		We show that this system includes the $q$-Appell--Lauricella system as a degeneration.
	\end{abstract}
	\section{Introduction}
	The Gauss hypergeometric function
	\begin{align}
		{}_2F_1\left(\begin{array}{c}
			\alpha,\beta\\
			\gamma
		\end{array};x\right)=1+\frac{\alpha\cdot\beta}{\gamma\cdot 1}x+\frac{\alpha(\alpha+1)\cdot\beta(\beta+1)}{\gamma(\gamma+1)\cdot 1\cdot 2}x^2+\cdots ,
	\end{align}
	has an integral representation:
	\begin{align}
		{}_2F_1\left(\begin{array}{c}
			\alpha,\beta\\
			\gamma
		\end{array};x\right)=\frac{\Gamma(\gamma)}{\Gamma(\alpha)\Gamma(\gamma-\alpha)}\int_0^1 t^{\alpha-1}(1-t)^{\gamma-\alpha-1}(1-xt)^{-\beta}dt.
	\end{align}
	This provides many global properties for the Gauss hypergeometric function.
	In the theory of hypergeometric functions, to give an integral representation is an important and fundamental problem.
	
	Kajihara's very-well-poised $q$-hypergeomeric series $W^{M,N}$ \cite{Kaji}, given in \eqref{defkaji}, is a multivariable $q$-hypergeometric function.
	This series is an extension of Holman's hypergeometric series \cite{holman,holman2} and Milne's hypergeometric series \cite{milne,Milne1988}.
	An elliptic analog of $W^{M,N}$ is also studied in \cite{KN}.
	The series $W^{M,N}$ has summation formulas and transformation formulas, and these formulas are applied in some integrable systems.
	See also \cite{kaji2018} for this series and its variations.
	On the other hand, the $q$-difference equation satisfied by Kajihara's $q$-hypergeometric series {has not been} obtained.
	
	In this paper, we give an integral representation for Kajihara's very well-poised $q$-hypergeometric series $W^{M,2}$ \cite{Kaji}.
	We also construct a $q$-difference system associated with this integral.
	Main results of this paper are Theorem \ref{thmkajiint} and \ref{thmmultisystem}.
	Theorem \ref{thmkajiint} gives a transformation formula between the series $W^{M,2}$ and the following Jackson integral:
	\begin{align}
		\int_{q/a_i}^{q/a_j}\prod_{k=1}^{M+3}\frac{(a_{k}t)_\infty}{(b_{k}t)_\infty}d_qt\ \ (a_1\cdots a_{M+3}=q^2b_1\cdots b_{M+3}).\label{introintegral}
	\end{align}  
	Notations in \eqref{introintegral} are summarized in section \ref{secpre}.
	{For $M=1$, our integral representation reduces to Bailey's formula, which express an ${}_8W_7$-series as a sum of two ${}_4\varphi_3$-series.}
	Theorem \ref{thmmultisystem} gives a multivariable $q$-difference system associated with the integral \eqref{introintegral}.
	This system is an extension of {a} variant of {the} $q$-hypergeometric equation of degree three $\mathcal{H}_3$ \cite{HMST}.
	It was found by Fujii and the author \cite{FN} that this variant can be regarded as a $q$-analog of the Riemann--Papperitz system, hence we call the system in Theorem \ref{thmmultisystem} $q$-$RP^M$.
	The Riemann--Papperitz system is a second order Fuchsian differential equation with three singularities $\{t_1,t_2,t_3\}$ on the Riemann sphere $\mathbb{P}^1=\mathbb{C}\cup\{\infty\}$, and this has solutions of the following forms:
	\begin{align}
		\label{RiePapint}&\mbox{(gauge factor)}\times\int_C (1-t_1t)^{\nu_1}(1-t_2t)^{\nu_2}(1-t_3t)^{\nu_3}(1-xt)^{\nu_4}dt\ \ (\nu_1+\nu_2+\nu_3+\nu_4=-2),\\
		\label{RiePapser}&\mbox{(gauge factor)}\times {}_2F_1\left(\begin{array}{c}
			\mu_1,\mu_2\\
			\mu_3
		\end{array};\frac{x-t_1}{x-t_3}\frac{t_2-t_3}{t_1-t_3}\right).
	\end{align}
	They \cite{FN} also gave integral solutions and series solutions, which are $q$-analogs of \eqref{RiePapint} and \eqref{RiePapser} respectively, for the variant $\mathcal{H}_3$.
	For integral solutions, see also \cite{AT}.
	Our result (Theorem \ref{thmmultisystem}) includes integral solutions for the variant $\mathcal{H}_3$ as a special case $M=1$.
	

	
	The contents of this paper are as follows.
	In section \ref{secpre}, we give notations and properties of  Kajihara's $q$-hypergeometric series $W^{M,N}$.
	In section \ref{secintsys}, we first derive an integral representation for the series $W^{M,2}$.
	We also construct a multivariable $q$-difference system $q$-$RP^M$ associated with this integral.
	In section \ref{secdegene}, we show that the system $q$-$RP^M$ degenerates to the $q$-Appell--Lauricella system.
	We summarize the results and discuss related problems in section \ref{secsum}.
	
	In this paper, we always assume that none of the denominators vanishes.

	\section{Preliminaries}\label{secpre}
	Throughout this paper, we fix $q\in\mathbb{C}$ with $0<|q|<1$, and we use the following notations:
	\begin{align}
		&(a)_\infty=\prod_{i=0}^\infty(1-aq^i),\ (a)_l=\frac{(a)_\infty}{(aq^l)_\infty},\ (a_1,\ldots,a_r)_l=(a_1)_l\cdots (a_r)_l,\\
		&\theta(x)=(x,q/x)_\infty,\\
		&\int_0^\tau f(t)d_qt=(1-q)\sum_{n=0}^\infty f(\tau q^n)\tau q^n,\ 
		\int_0^{\tau\infty}f(t)d_qt=(1-q)\sum_{n=-\infty}^\infty f(\tau q^n)\tau q^n,\\
		&\int_{\tau_1}^{\tau_2}f(t)d_qt=\int_{0}^{\tau_2}f(t)d_qt-\int_{0}^{\tau_1}f(t)d_qt,\\
		&T_x f(x)=f(qx),\\
		&{}_r\varphi_s\left(\begin{array}{c}
			a_1,a_2,\ldots,a_r\\b_1,b_2,\ldots,b_s
		\end{array};z\right)=\sum_{l=0}^\infty \frac{(a_1,a_2,\ldots,a_r)_l}{(q,b_1,b_2,\ldots,b_s)_l}\left[(-1)^l q^{\binom{l}{2}}\right]^{1+s-r} z^l.
	\end{align}
	\begin{defn}[\cite{GR}]
		The very well-poised $q$-hypergeometric series ${}_{r+1}W_r$ is defined as follows:
		\begin{align}
			\label{defvwp}{}_{r+1}W_r(a_1;a_4,a_5,\ldots,a_{r+1};z)=\sum_{l=0}^\infty\frac{1-a_1 q^{2l}}{1-a_1}\frac{(a_1,a_4,a_5,\ldots,a_{r+1})_l}{(q,qa_1/a_4,qa_1/a_5,\ldots,qa_1/a_{r+1})_l}z^l.
		\end{align}
	\end{defn}
	For properties of ${}_{r+1}W_r$, see Gasper--Rahman \cite{GR}.
	Kajihara's $q$-hypergeometric series $W^{M,N}$ \cite{Kaji} is a multiple extension of ${}_{r+1}W_r$ {defined} as follows:
	\begin{defn}[Kajihara's $q$-hypergeometric series \cite{Kaji,kaji2018,KN}]
		The $q$-hypergeometric series $W^{M,N}$ is defined as follows: 
		\begin{align}
			\notag&W^{M,N}(\{x_i\}_{1\leq i\leq M};a;\{u_j\}_{1\leq j\leq M+N};\{v_k\}_{1\leq k\leq N};z)\\
			\label{defkaji}&=\sum_{l\in\mathbb{Z}_{\geq0}^M}z^{|l|}\frac{\Delta(x q^l)}{\Delta(x)}\prod_{i=1}^M\left(\frac{1-ax_iq^{|l|+l_i}}{1-ax_i}\frac{(ax_i)_{|l|}\prod_{j=1}^{M+N}(x_iu_j)_{l_i}}{\prod_{j=1}^M(qx_i/x_j)_{l_i}\prod_{k=1}^N(aqx_i/v_k)_{l_i}}\right)\frac{\prod_{j=1}^N(v_k)_{|l|}}{\prod_{j=1}^{M+N}(aq/u_j)_{|l|}}
		\end{align}
		where $|l|=l_1+l_2+\cdots+l_M$, $x q^l=\{x_i q^{l_i}\}_{1\leq i\leq M}$ and $\Delta(x)$ is the Vandermonde determinant of $x$:
		\begin{align}
			\Delta(x)=\prod_{1\leq i<j\leq M}(x_i-x_j).
		\end{align}
		The function $W^{M,N}$ is Kajihara's very well-poised $q$-hypergeometric series.
	\end{defn}
	\begin{rem}
		In \cite{Kaji,kaji2018,KN}, the series $W^{M,N}$ is originally given in the following form:
		\begin{align}
						\notag&W^{M,N}\left(\begin{array}{c}
								\{a_i\}_{1\leq i\leq M}\\
								\{x_i\}_{1\leq i\leq M}
							\end{array}\bigg| \, s;
							\{u_k\}_{1\leq k\leq N};
							\{v_k\}_{1\leq k\leq N}
						;z\right)\\
						\notag&=\sum_{l\in\mathbb{Z}_{\geq0}^M}z^{|l|}\frac{\Delta(x q^l)}{\Delta(x)}\prod_{i=1}^M \frac{1-q^{|l|+l_i}sx_i}{1-sx_i}\\
						&\phantom{=}\times \prod_{j=1}^M\left[\frac{(sx_j)_{|l|}}{(qsx_j/a_j)_{|l|}}\prod_{i=1}^M\frac{(a_jx_i/x_j)_{l_i}}{(qx_i/x_j)_{l_i}}\right]\cdot\prod_{k=1}^N \left[\frac{(v_k)_{|l|}}{(qs/u_k)_{|l|}}\prod_{i=1}^M\frac{(x_iu_k)_{l_i}}{(qsx_i/v_k)_{l_i}}\right].
		\end{align}	
		Putting $a_i=x_iu_{N+i}$ and $s=a$, this series becomes \eqref{defkaji}.
		The symmetry for $u_1,\ldots,u_{M+N}$ is obvious in the notation \eqref{defkaji}.
		Hence we use the notation \eqref{defkaji} in this paper. 
	\end{rem}
	\begin{rem}
		{In this paper we will discuss the special case $W^{M,2}$.
		This series first appeared in Milne's work \cite{Milne1988} in the context of $U(n)$ generalization for several summation formula of $q$-hypergeometric series.}
	\end{rem}
	
	{For non-terminating case, the series \eqref{defkaji} converges if $|z|<1$.
	This can be checked by the multiple power series ratio test (see Appendix \ref{appA}).}
	
	The series $W^{M,N}$ has an transformation formula.
	For the proof, see \cite{Kaji,KN}.
	\begin{prop}[Kajihara's transformation formula \cite{Kaji,kaji2018,KN}]
		For any $n\in\mathbb{Z}_{\geq 0}$, we have
		\begin{align}
			\notag&W^{M,N+2}(\{x_i\}_{1\leq i\leq M};a;\{b_j\}_{1\leq j\leq M+N+2};\{c/y_k\}_{1\leq k\leq N},\mu c q^n,q^{-n};q)\\
			\notag&=\prod_{i=1}^M \frac{(aqx_i)_n}{(\mu c/(a x_i))_n}\prod_{j=1}^{M+N+2}\frac{(\mu c b_j/a)_n}{(a q/b_j)_n}\prod_{k=1}^N\frac{(c/y_k)_n}{(\mu q y_k)_n}
			\\
			\label{Kajitrans}&\phantom{=}\times W^{N,M+2}(\{y_k\}_{1\leq k\leq N};\mu;\{aq/(cb_j)\}_{1\leq j\leq M+N+2};\{\mu c/(a x_i)\}_{1\leq i\leq M},\mu cq^n,q^{-n};q)
		\end{align}
		where {$\mu=a^{N+2} q^{N+1}y_1 y_2 \cdots y_N/(c^{N+1} b_1 b_2 \cdots b_{M+N+2} x_1 x_2 \cdots x_M )$.}
	\end{prop}
	\begin{cor}[\cite{Kaji}]
		We have
		\begin{align}
			\notag&W^{M,3}(\{x_i\}_{1\leq i\leq M};a;\{b_j\}_{1\leq j\leq M+3};c,\mu c q^k,q^{-k};q)\\
			\notag&=\prod_{i=1}^M \frac{(aqx_i)_k}{(\mu c/(a x_i))_k}\prod_{j=1}^{M+3}\frac{(\mu c b_j/a)_k}{(a q/b_j)_k}\cdot \frac{(c)_k}{(\mu q )_k}
			\\
			\label{transWn3}&\phantom{=}\times {}_{2M+8}W_{2M+7}(\mu;\{\mu c/(ax_i)\}_{1\leq i\leq M},\{aq/(cb_j)\}_{1\leq j\leq M+3},\mu c q^k,q^{-k};q)
		\end{align}
		where {$\mu=a^{3} q^{2}/(c^{2} b_1 b_2 \cdots b_{M+3} x_1 x_2 \cdots x_M )$} and $k\in\mathbb{Z}_{\geq0}$.
	\end{cor}
	\begin{proof}
		We get the desired equation by putting $N=1$ in Kajihara's  transformation formula \eqref{Kajitrans}.
	\end{proof}
	\begin{rem}
		When $M=1$, the formula \eqref{transWn3} reduces {to} a transformation formula of terminating very well-poised $q$-hypergeometric series ${}_{10}W_9$ (see Gasper--Rahman \cite{GR} and references therein).
	\end{rem}

\section{Integral representation for $W^{M,2}$ and corresponding $q$-difference system}\label{secintsys}

In this section, we derive  an integral representation for Kajihara's $q$-hypergeometric series $W^{M,2}$ by using \eqref{transWn3}.
We also construct a $q$-difference system associated with this integral.

\subsection{Integral representation for $W^{M,2}$}
In this subsection, we give an integral representation for $W^{M,2}$.
This formula is an extension of Bailey's transformation formula \cite[(2.10.{19})]{GR}:
	\begin{align}
			\notag&\int_{a}^{b}\frac{(qt/a,qt/b,ct,dt)_{\infty}}{(et,ft,gt,ht)_{\infty}}d_{q}t\\
			=&b(1-q)\frac{(q,bq/a,a/b,cd/(eh),cd/(fh),cd/(gh),bc,bd)_{\infty}}{(ae,af,ag,be,bf,bg,bh,bcd/h)_{\infty}}
			\times {}_{8}W_{7}\left(\frac{bcd}{hq};be,bf,bg,\frac{c}{h},\frac{d}{h};ah\right).\label{inttoser}
		\end{align}
	where $cd=abefgh$.
	{The formula \eqref{inttoser} is also known in the following form \cite[(2.10.10)]{GR}:
	\begin{align}
		&{}_8W_7\left(a;b,c,d,e,f;\frac{a^2q^2}{bcdef}\right)
		\notag\\
		&=\frac{(aq,aq/(de),aq/(df),aq/(ef))_\infty}{(aq/d,aq/e,aq/f,aq/(def))_\infty}{}_4\varphi_3\left(\begin{array}{c}
			aq/(bc),d,e,f\\aq/b,aq/e,def/a
		\end{array};q\right)
		\notag\\
		&+\frac{(aq,aq/(bc),d,e,f,a^2q^2/(bdef),a^2q^2/(cdef))_\infty}{(aq/b,aq/c,aq/d,aq/e,aq/f,a^2q^2/(bcdef),def/(aq))_\infty}
		\notag\\
		&\phantom{+ }\times{}_4\varphi_3\left(\begin{array}{c}
			aq/(de),aq/(df),aq/(ef),a^2q^2/(bcdef)\\ a^2q^2/(bdef),a^2q^2/(cdef),aq^2/(def)
		\end{array};q\right).
	\end{align}}
\begin{thm}\label{thmkajiint}
	{We put $d=a^2q^2/(c_1c_2b_1b_2\cdots b_{M+2}x_1x_2\cdots x_M)$.}
	We suppose $|d|<1$.
	We have
	\begin{align}
		\notag&W^{M,2}(\{x_i\}_{1\leq i\leq M};a;\{b_j\}_{1\leq j\leq M+2};c_1,c_2;d)	\\	
		\notag&=\prod_{i=1}^M \frac{(aqx_i)_\infty}{(a qx_i/c_1)_\infty}\prod_{j=1}^{M+2}\frac{(aq/( c_1 b_j))_\infty}{(a q/b_j)_\infty}\cdot \frac{(c_1d,c_2)_\infty}{(d,c_2/c_1)_\infty}\\
		\notag&\phantom{=}\times{}_{M+3}\varphi_{M+2}\left(\begin{array}{c}
			c_1,\{aq/(c_2 b_j)\}_{1\leq j\leq M+2}\\
			qc_1/c_2,c_1d,\{aqx_i/c_2\}_{1\leq i\leq M}
		\end{array};q\right)\\
		&+\mathrm{idem}(c_1;c_2),\label{newWM2}
	\end{align}
	{where the symbol ``$\mathrm{idem}(c_1;c_2)$'' after an expression means that the expression is repeated with $c_1$ and $c_2$ interchanged.}
	{Equivalently, we have
	\begin{align}
		\notag&\int_{q/a_{M+2}}^{q/a_{M+3}}\prod_{k=1}^{M+3}\frac{(a_k t )_\infty}{(b_k t )_\infty}d_qt\\
		\notag&=
		\prod_{i=1}^M\frac{(qa_i/a_{M+2},qa_i/a_{M+3})_\infty}{(q^2a_ib_{M+3}/(a_{M+2}a_{M+3}))_\infty}\prod_{j=1}^{M+2}{(q^2b_jb_{M+3}/(a_{M+2}a_{M+3}))_\infty}\prod_{j=1}^{M+3}\frac{1}{(qb_j/a_{M+2},qb_j/a_{M+3})_\infty}\\
		\notag&\phantom{=}\times {(q,a_{M+1}/b_{M+3},a_{M+2}/a_{M+3},a_{M+3}/a_{M+2})_\infty}\frac{(1-q)q}{a_{M+3}-a_{M+2}}\\
		&\phantom{=}\times W^{M,2}\left(\{a_i\}_{1\leq i\leq M};\frac{qb_{M+3}}{a_{M+2}a_{M+3}};\left\{\frac{1}{b_j}\right\}_{1\leq j\leq M+2};\frac{qb_{M+3}}{a_{M+2}},\frac{qb_{M+3}}{a_{M+3}};\frac{a_{M+1}}{b_{M+3}}\right),\label{inttoWpart}
		%
	\end{align}
	where $a_1\cdots a_{M+3}=q^2 b_1\cdots b_{M+3}$ and $|a_{M+1}/b_{M+3}|<1$.}
\end{thm}
Before proving this theorem, we give two remarks.
\begin{rem}
	When $M=1$, by putting 
	\begin{align}
		a_1=c,\ a_2=d,\ a_3=q/a,\ a_4=q/b,\ b_1=f,\ b_2=g,\ b_3=h,\ b_4=e,
	\end{align}
	the formula \eqref{inttoWpart} reduces to
	\begin{align}
		&\int_{a}^{b}\frac{(qt/a,qt/b,ct,dt)_\infty}{(et,ft,gt,ht)_\infty}d_qt\\
		&=b(1-q)\frac{(q,a/b,bq/a,ac,bc,d/e,abef,abeg,abeh)_\infty}{(ae,af,ag,ah,be,bf,bg,bh,abce)_\infty}{}_8W_7\left(\frac{abce}{q};ae,\frac{c}{f},\frac{c}{g},be,\frac{c}{h};\frac{d}{e}\right),
	\end{align}
	where $cd=abefgh$.
	Applying the transformation formula \cite[(2.10.1)]{GR}:
	\begin{align}
		\label{87trans}{}_8W_7\left(a;b,c,d,e,f;\frac{a^2q^2}{bcdef}\right)=\frac{(a q,a q/(e f),\lambda q/e,\lambda q/f)_\infty}{(a q/e,a q/f,\lambda q,\lambda q/(e f))_\infty}{}_8W_7\left(\lambda;\frac{\lambda b}{a},\frac{\lambda c}{a},\frac{\lambda d}{a},e,f;\frac{aq}{ef}\right),
	\end{align}
	where $\lambda =q a^2/(b c d)$, we have \eqref{inttoser}.
	We note that the transformation \eqref{87trans} can be derived from \eqref{inttoWpart} (see Remark \ref{remM1trans}).
\end{rem}
\begin{rem}
	Another transformation formula for $W^{M,2}$ is studied by Milne--Newcomb \cite[Theorem 4.1]{MN2012}.
	They give a transformation between $W^{M,2}$ and some $M$-ple sum.
\end{rem}
\begin{proof}[proof of Theorem \ref{thmkajiint}]
	{The proof is based on the method similar to Bailey's one \cite{B}.
		For Bailey's idea, see Gasper--Rahman \cite[section 2]{GR}.}
	The formula {\eqref{newWM2}} is derived by considering the limit $k\to\infty$ in \eqref{transWn3}.
	First, we consider the limit
	\begin{align}
		\lim_{k\to\infty}W^{M,3}(\{x_i\}_{1\leq i\leq M};a;\{b_j\}_{1\leq j\leq M+2},b_{M+3} q^k;c,\mu c ,q^{-k};q).
	\end{align}
	We recall the definition of the function $W^{M,3}$:
	\begin{align}
		\notag
		&W^{M,3}(\{x_i\}_{1\leq i\leq M};a;\{b_j\}_{1\leq j\leq M+2},b_{M+3} q^k;c,\mu c ,q^{-k};q)
		\\
		\notag
		&\sum_{l\in\mathbb{Z}_{\geq0}^M}q^{|l|}\frac{\Delta(xq^l)}{\Delta(x)}\prod_{i=1}^M\left(\frac{1-ax_iq^{|l|+l_i}}{1-ax_i}\frac{(ax_i)_{|l|}\prod_{j=1}^{M+2}(x_ib_j)_{l_i}}{\prod_{j=1}^M (qx_i/x_j)_{l_i}\cdot(aqx_i/c,aqx_i/(\mu c))_{l_i}}\right)\frac{(c,\mu c)_{|l|}}{\prod_{j=1}^{M+2}(aq/b_j)_{|l|}}
		\\
		&\phantom{\sum_{l\in\mathbb{Z}_{\geq0}^M}q^{|l}}\times \prod_{i=1}^M \frac{(x_ib_{M+3} q^k)_{l_i}}{(aqx_iq^k)_{l_i}}\cdot \frac{(q^{-k})_{|l|}}{(aqq^{-k}/b_{M+3})_{|l|}}.
	\end{align}
	Since $(x_ib_{M+3}q^k)_{l_i}$, $(a q x_i q^k)_{l_i}\xrightarrow{k\to\infty}1$ and $\displaystyle\frac{(q^{-k})_{|l|}}{(aqq^{-k}/b_{M+3})_{|l|}}\xrightarrow{k\to\infty}\left(\frac{b_{M+3}}{aq}\right)^{|l|}$, we have
	\begin{align}
		\notag
		&\lim_{k\to\infty}W^{M,3}(\{x_i\}_{1\leq i\leq M};a;\{b_j\}_{1\leq j\leq M+2},b_{M+3} q^k;c,\mu c ,q^{-k};q)
		\\
		\notag
		&\sum_{l\in\mathbb{Z}_{\geq0}^M}\left(\frac{b_{M+3}}{a}\right)^{|l|}\frac{\Delta(xq^l)}{\Delta(x)}\prod_{i=1}^M\left(\frac{1-ax_iq^{|l|+l_i}}{1-ax_i}\frac{(ax_i)_{|l|}\prod_{j=1}^{M+2}(x_ib_j)_{l_i}}{\prod_{j=1}^M (qx_i/x_j)_{l_i}\cdot(aqx_i/c,aqx_i/(\mu c))_{l_i}}\right)\frac{(c,\mu c)_{|l|}}{\prod_{j=1}^{M+2}(aq/b_j)_{|l|}}
		\\
		&=W^{M,2}\left(\{x_i\}_{1\leq i\leq M};a;\{b_j\}_{1\leq j\leq M+2};c,\mu c;\frac{b_{M+3}}{a}\right).\label{just1}
	\end{align}
	{Here, we take the termwise limit.
	This is justified by Tannery's theorem if $|b_{M+3}/a|<1$ (see Appendix \ref{appA}).}

	On the other hand, if $\mu=a^{3} q^{2}/(c^{2} b_1 b_2 \cdots b_{M+3} x_1 x_2 \cdots x_M )$, we have
	\begin{align}
		\notag&W^{M,3}(\{x_i\}_{1\leq i\leq M};a;\{b_j\}_{1\leq j\leq M+2},b_{M+3}q^k;c,\mu c q^k,q^{-k};q)\\
		\notag&=\prod_{i=1}^M \frac{(aqx_i)_k}{(\mu c q^{-k}/(a x_i))_k}\prod_{j=1}^{M+2}\frac{(\mu c b_jq^{-k}/a)_k}{(a q/b_j)_k}\cdot \frac{(\mu c b_{M+3}/a,c)_k}{(a q q^{-k}/b_{M+3},\mu q q^{-k})_k}
		\\
		&\phantom{=}\times {}_{2M+8}W_{2M+7}(\mu q^{-k};\{\mu c q^{-k}/(ax_i)\}_{1\leq i\leq M},\{aq/(cb_j)\}_{1\leq j\leq M+2},aqq^{-k}/(cb_{M+3}),\mu c,q^{-k};q)
	\end{align}
	by putting $b_{M+3}\to b_{M+3} q^k$ in \eqref{transWn3}.
	Since $(A)_k\xrightarrow{k\to\infty}(A)_\infty$ and 
	\begin{align}
		\prod_{l=1}^L \frac{(A_i q^{-k})_k}{(B_i q^{-k})_k}=\prod_{l=1}^{L}\left(\frac{A_i}{B_i}\right)^k\frac{(q/A_i)_k}{(q/B_i)_k}\xrightarrow{k\to\infty}\prod_{l=1}^L \frac{(q/A_i)_\infty}{(q/B_i)_\infty},
	\end{align}
	if $A_1\cdots A_L=B_1\cdots B_L$, we have
	\begin{align}
		\notag&\prod_{i=1}^M \frac{(aqx_i)_k}{(\mu c q^{-k}/(a x_i))_k}\prod_{j=1}^{M+2}\frac{(\mu c b_jq^{-k}/a)_k}{(a q/b_j)_k}\cdot \frac{(\mu c b_{M+3}/a,c)_k}{(a q q^{-k}/b_{M+3},\mu q q^{-k})_k}
		\\
		&\xrightarrow{k\to\infty}\prod_{i=1}^M \frac{(aqx_i)_\infty}{(a qx_i/(\mu c))_\infty}\prod_{j=1}^{M+2}\frac{(aq/(\mu c b_j))_\infty}{(a q/b_j)_\infty}\cdot \frac{(\mu c b_{M+3}/a,c)_\infty}{(b_{M+3}/a,1/\mu)_\infty}
	\end{align}
	We put $k=2K-1$ and
	\begin{align}
		C_l=\frac{1-\mu q^{-k}q^{2l}}{1-\mu q^{-k}}\prod_{i=1}^M\frac{(\mu cq^{-k}/(ax_i))_l}{(aqx_i/c)_l}\prod_{j=1}^{M+2}\frac{(aq/(cb_j))_l}{(\mu cb_jq^{-k}/a)_l}\cdot\frac{(\mu q^{-k},a q q^{-k}/(c b_{M+3}),\mu c,q^{-k})_l}{(q,\mu c b_{M+3}/a,q q^{-k}/c,q \mu)_l}q^l
	\end{align}
	By the definition of the very well-poised $q$-hypergeometric series ${}_{2M+8}W_{2M+7}$, we have
	\begin{align}
		\notag&{}_{2M+8}W_{2M+7}(\mu q^{-k};\{\mu c q^{-k}/(ax_i)\}_{1\leq i\leq M},\{aq/(cb_j)\}_{1\leq j\leq M+2},aqq^{-k}/(cb_{M+3}),\mu c,q^{-k};q)\\
		&=\sum_{l=0}^k C_l=\sum_{l=0}^{K-1}C_l+\sum_{l=0}^{K-1}C_{k-l}.
	\end{align}
	For $0\leq l\leq K-1$, we have
	\begin{align}
		C_l\xrightarrow{K\to\infty}\frac{\prod_{j=1}^{M+2}(aq/(cb_j))_l\cdot (\mu c)_l}{\prod_{i=1}^M(aqx_i/c)_l\cdot (q,\mu c b_{M+3}/a,q\mu)_l} q^l,
	\end{align}
	and 
	\begin{align}
		\notag&C_{k-l}\xrightarrow{K\to\infty}\frac{(c)_l\prod_{j=1}^{M+2}(aq/(\mu c b_j))}{\prod_{i=1}^M(aqx_i/(\mu c))_l\cdot (q,q/\mu,c b_{M+3}/a)_l}
		q^l\\
		&\times\prod_{i=1}^M\frac{(aqx_i/(\mu c))_\infty}{(aqx_i/c)_\infty}\prod_{j=1}^{M+2}\frac{(aq/(cb_j))_\infty}{(aq/(\mu c b_j))_\infty}\cdot\frac{(1/\mu,\mu c,cb_{M+3}/a)_\infty}{(c,\mu,\mu c b_{M+3}/a)_\infty}
	\end{align}
	The limit of $C_{k-l}$ can be calculated by using formulas $(a)_{k-l}=(a)_{k}/(a q^{k-l})_l$, $(aq^{-l})_l=(-a)^l q^{-l(l-1)/2}(q/a)_l$, $(a)_\infty=(aq)_\infty(1-a)$ and the condition $\mu=a^{3} q^{2}/(c^{2} b_1 b_2 \cdots b_{M+3} x_1 x_2 \cdots x_M )$.
	Hence, we have
	\begin{align}
		\notag\lim_{K\to\infty}\sum_{l=0}^{K-1}C_l&=\sum_{l=0}^\infty \frac{\prod_{j=1}^{M+2}(aq/(cb_j))_l\cdot (\mu c)_l}{\prod_{i=1}^M(aqx_i/c)_l\cdot (q,\mu c b_{M+3}/a,q\mu)_l} q^l,\\
		&={}_{M+3}\varphi_{M+2} \left(\begin{array}{c}
			\mu c,\{aq/(cb_j)\}_{1\leq j\leq M+2}\\\mu c b_{M+3}/a,q\mu,\{aqx_i/c\}_{1\leq i\leq M}
		\end{array};q\right),\label{just2}
	\end{align}
	and 
	\begin{align}
		\notag\lim_{K\to\infty}\sum_{l=0}^{K-1}C_{k-l}&=\prod_{i=1}^M\frac{(aqx_i/(\mu c))_\infty}{(aqx_i/c)_\infty}\prod_{j=1}^{M+2}\frac{(aq/(cb_j))_\infty}{(aq/(\mu c b_j))_\infty}\cdot\frac{(1/\mu,\mu c,cb_{M+3}/a)_\infty}{(c,\mu,\mu c b_{M+3}/a)_\infty}\\
		\notag&\phantom{=}\times\sum_{l=0}^\infty\frac{(c)_l\prod_{j=1}^{M+2}(aq/(\mu c b_j))}{\prod_{i=1}^M(aqx_i/(\mu c))_l\cdot (q,q/\mu,c b_{M+3}/a)_l}q^l
		\\
		\notag&=\prod_{i=1}^M\frac{(aqx_i/(\mu c))_\infty}{(aqx_i/c)_\infty}\prod_{j=1}^{M+2}\frac{(aq/(cb_j))_\infty}{(aq/(\mu c b_j))_\infty}\cdot\frac{(1/\mu,\mu c,cb_{M+3}/a)_\infty}{(c,\mu,\mu c b_{M+3}/a)_\infty}\\
		&\phantom{=}\times{}_{M+3}\varphi_{M+2}\left(\begin{array}{c}
			c,\{aq/(\mu c b_j)\}_{1\leq j\leq M+2}\\
			q/\mu,cb_{M+3}/a,\{aqx_i/(\mu c)\}_{1\leq i\leq M}
		\end{array};q\right).\label{just3}
	\end{align}
	{These limits are also justified by Tannery's theorem (see Appendix \ref{appA}).}
	{In conclusion, we obtain}
	\begin{align}
		\notag&W^{M,2}\left(\{x_i\}_{1\leq i\leq M};a;\{b_j\}_{1\leq j\leq M+2};c,\mu c;\frac{b_{M+3}}{a}\right)\\
		\notag&=\prod_{i=1}^M \frac{(aqx_i)_\infty}{(a qx_i/(\mu c))_\infty}\prod_{j=1}^{M+2}\frac{(aq/(\mu c b_j))_\infty}{(a q/b_j)_\infty}\cdot \frac{(\mu c b_{M+3}/a,c)_\infty}{(b_{M+3}/a,1/\mu)_\infty}\\
		\notag&\phantom{=}\times{}_{M+3}\varphi_{M+2} \left(\begin{array}{c}
			\mu c,\{aq/(cb_j)\}_{1\leq j\leq M+2}\\\mu c b_{M+3}/a,q\mu,\{aqx_i/c\}_{1\leq i\leq M}
		\end{array};q\right)\\
		\notag&+\prod_{i=1}^M \frac{(aqx_i)_\infty}{(a qx_i/c)_\infty}\prod_{j=1}^{M+2}\frac{(aq/( c b_j))_\infty}{(a q/b_j)_\infty}\cdot \frac{(\mu c,cb_{M+3}/a)_\infty}{(b_{M+3}/a,\mu)_\infty}\\
		&\phantom{=}\times{}_{M+3}\varphi_{M+2}\left(\begin{array}{c}
			c,\{aq/(\mu c b_j)\}_{1\leq j\leq M+2}\\
			q/\mu,cb_{M+3}/a,\{aqx_i/(\mu c)\}_{1\leq i\leq M}
		\end{array};q\right).
	\end{align}
	Putting $c=c_1$, $b_{M+3}=a^3q^2/(c_1 c_2b_1b_2\cdots b_{M+2}x_1x_2\cdots x_M)$, $d=a^2q^2/(c_1c_2b_1b_2\cdots b_{M+2}x_1x_2\cdots x_M)$, we finally get the desired equation \eqref{newWM2}.
	The equation \eqref{inttoWpart} is obtained by \eqref{newWM2}.
	{By simple calculations, each sum in the right-hand side of \eqref{newWM2} can be rewritten by Jackson integral as follows:}
	\begin{align}
		\notag&=\prod_{i=1}^M \frac{(aqx_i)_\infty}{(a qx_i/c_1)_\infty}\prod_{j=1}^{M+2}\frac{(aq/( c_1 b_j))_\infty}{(a q/b_j)_\infty}\cdot \frac{(c_1d,c_2)_\infty}{(d,c_2/c_1)_\infty}\\
		\notag&\phantom{=}\times{}_{M+3}\varphi_{M+2}\left(\begin{array}{c}
			c_1,\{aq/(c_2 b_j)\}_{1\leq j\leq M+2}\\
			qc_1/c_2,c_1d,\{aqx_i/c_2\}_{1\leq i\leq M}
		\end{array};q\right),\\
		\notag&=\prod_{i=1}^M \frac{(aqx_i)_\infty}{(a qx_i/c_1,aqx_i/c_2)_\infty}\prod_{j=1}^{M+2}\frac{(aq/( c_1 b_j),aq/(c_2b_j))_\infty}{(a q/b_j)_\infty}\cdot \frac{(c_1,c_2)_\infty}{(q,d,c_1/c_2,c_2/c_1)_\infty}\\
		&\phantom{=}\times\frac{c_2-c_1}{1-q}\int_0^{1/c_2} \frac{(qc_1t,qc_2t,c_1c_2dt)_\infty\prod_{i=1}^M(aqx_it)_\infty}{(c_1c_2t)_\infty\prod_{j=1}^{M+2}(aqt/b_j)_\infty} d_qt,\\
		\notag&\prod_{i=1}^M \frac{(aqx_i)_\infty}{(a qx_i/c_2)_\infty}\prod_{j=1}^{M+2}\frac{(aq/(c_2 b_j))_\infty}{(a q/b_j)_\infty}\cdot \frac{(c_2 d,c_1)_\infty}{(d,c_1/c_2)_\infty}\\
		\notag&\phantom{=}\times{}_{M+3}\varphi_{M+2} \left(\begin{array}{c}
			c_2,\{aq/(c_1b_j)\}_{1\leq j\leq M+2}\\qc_2/c_1,c_2d,\{aqx_i/c_1\}_{1\leq i\leq M}
		\end{array};q\right)\\
		\notag&=\prod_{i=1}^M \frac{(aqx_i)_\infty}{(a qx_i/c_1,aqx_i/c_2)_\infty}\prod_{j=1}^{M+2}\frac{(aq/( c_1 b_j),aq/(c_2b_j))_\infty}{(a q/b_j)_\infty}\cdot \frac{(c_1,c_2)_\infty}{(q,d,c_1/c_2,c_2/c_1)_\infty}\\
		&\phantom{=}\times\frac{c_1-c_2}{1-q}\int_0^{1/c_1} \frac{(qc_1t,qc_2t,c_1c_2dt)_\infty\prod_{i=1}^M(aqx_it)_\infty}{(c_1c_2t)_\infty\prod_{j=1}^{M+2}(aqt/b_j)_\infty} d_qt.
	\end{align} 
	Therefore we have
	\begin{align}
		\notag&W^{M,2}(\{x_i\}_{1\leq i\leq M};a;\{b_j\}_{1\leq j\leq M+2};c_1,c_2;d)\\
		\notag&=\prod_{i=1}^M \frac{(aqx_i)_\infty}{(a qx_i/c_1,aqx_i/c_2)_\infty}\prod_{j=1}^{M+2}\frac{(aq/( c_1 b_j),aq/(c_2b_j))_\infty}{(a q/b_j)_\infty}\cdot \frac{(c_1,c_2)_\infty}{(q,d,c_1/c_2,c_2/c_1)_\infty}\\
		&\phantom{=}\times\frac{c_2-c_1}{1-q}\int_{1/c_1}^{1/c_2} \frac{(qc_1t,qc_2t,c_1c_2dt)_\infty\prod_{i=1}^M(aqx_it)_\infty}{(c_1c_2t)_\infty\prod_{j=1}^{M+2}(aqt/b_j)_\infty} d_qt,
	\end{align}
	{where $d=a^2q^2/(c_1c_2b_1b_2\cdots b_{M+2}x_1x_2\cdots x_M)$.
	We put $a=q B_{M+3}/(A_{M+2}A_{M+3})$, $b_j=1/B_j$, $c_1=q B_{M+3}/A_{M+2}$, $c_2=q B_{M+3}/A_{M+3}$, $x_i=A_i$, rewrite $A_i$, $B_i$ as $a_i$, $b_i$, and change the variable of integration $t$ as $t a_{M+2}a_{M+3}/(q^2b_{M+3})$, then we have
	\begin{align}
		\notag&W^{M,2}\left(\{a_i\}_{1\leq i\leq M};\frac{qb_{M+3}}{a_{M+2}a_{M+3}};\left\{\frac{1}{b_j}\right\}_{1\leq j\leq M+2};\frac{qb_{M+3}}{a_{M+2}},\frac{qb_{M+3}}{a_{M+3}};\frac{a_{M+1}}{b_{M+3}}\right)\\
		\notag&=\prod_{i=1}^M\frac{(q^2a_ib_{M+3}/(a_{M+2}a_{M+3}))_\infty}{(qa_i/a_{M+2},qa_i/a_{M+3})_\infty}\prod_{j=1}^{M+2}\frac{(qb_j/a_{M+2},qb_j/a_{M+3})_\infty}{(q^2b_jb_{M+3}/(a_{M+2}a_{M+3}))_\infty} \\
		&\phantom{=}\times \frac{(qb_{M+3}/a_{M+2},qb_{M+3}/a_{M+3})_\infty}{(q,a_{M+1}/b_{M+3},a_{M+2}/a_{M+3},a_{M+3}/a_{M+2})_\infty}\frac{a_{M+3}-a_{M+2}}{(1-q)q}\int_{q/a_{M+2}}^{q/a_{M+3}}\prod_{k=1}^{M+3}\frac{(a_k t )_\infty}{(b_k t )_\infty}d_qt
	\end{align}
	where $a_1a_2\cdots a_{M+3}=q^2b_1b_2\cdots b_{M+3}$.}
 	This completes the proof of \eqref{inttoWpart}.
\end{proof}
\begin{cor}\label{corkajilinear}
	{
	We put $d=a^2q^2/(c_1c_2b_1b_2\cdots b_{M+2}x_1x_2\cdots x_M)$.
	\begin{itemize}
		\item[(1).] We have
		\begin{align}
			\notag&W^{M,2}(\{x_i\}_{1\leq i\leq M};a;\{b_j\}_{1\leq j\leq M+2};c_1,c_2;d)\\
			&=\frac{(c_1d,c_2d,aqx_1,aqx_1/(c_1c_2))_\infty}{(d,c_1c_2d,aqx_1/c_1,aqx_1/c_2)_\infty}W^{M,2}\left(\frac{c_1c_2d}{a q},\{x_i\}_{2\leq i\leq M};a;\{b_j\}_{1\leq j\leq M+2};c_1,c_2;\frac{aqx_1}{c_1c_2}\right),\label{eqsyma}
		\end{align}
		if $|d|<1$ and $|aqx_1/(c_1c_2)|<1$.
		\item[(2).] We have
		\begin{align}
			\notag&W^{M,2}(\{x_i\}_{1\leq i\leq M};a;\{b_j\}_{1\leq j\leq M+2};c_1,c_2;d)\\
			\notag&=\frac{(c_1c_2b_1d/(aq))_\infty}{(d)_\infty}\prod_{i=1}^M\frac{(aqx_i)_\infty}{(a^2q^2x_i/(c_1c_2b_1))_\infty}\prod_{j=2}^{M+2}\frac{(a^2q^2/(c_1c_2b_1b_j))_\infty}{(aq/b_j)_\infty}\\
			&\phantom{=}\times W^{M,2}\left(\{x_i\}_{1\leq i\leq M};\frac{a^2q}{c_1c_2b_1};\frac{aq}{c_1c_2},\{b_j\}_{2\leq j\leq M+2};\frac{aq}{c_1b_1},\frac{aq}{c_2b_1};\frac{c_1c_2b_1d}{aq}\right),\label{eqsymb}
		\end{align}
		if $|d|<1$ and $|c_1c_2b_1d/(aq)|<1$.
		\item[(3).] A three-term relation
		\begin{align}
			\notag&W^{M,2}(\{x_i\}_{1\leq i\leq M};a;\{b_j\}_{1\leq j\leq M+2};c_1,c_2;d)\\
			\notag&=\frac{(c_1,q/c_1,c_2d,q/(c_2d))_\infty}{(d,q/d,c_1/c_2,qc_2/c_1)_\infty}\prod_{i=1}^M\frac{(aqx_i,aq^2/(c_1c_2d))_\infty}{(aqx_i/c_2,aq^2/(c_1d))_\infty}\prod_{j=1}^{M+2}\frac{(aq/(c_2b_j),aq^2/(c_1db_j))_\infty}{(aq/b_j,aq^2/(c_1c_2db_j))_\infty}\\
			\notag&\phantom{=}\times  W^{M,2}\left(\{x_i\}_{1\leq i\leq M};\frac{aq}{c_1d};\{b_j\}_{1\leq j\leq M+2};c_2,\frac{q}{d};\frac{q}{c_1}\right)\\
			&+\mathrm{idem}(c_1;c_2).\label{threeterm}
		\end{align}
		holds if $|d|<1$, $|q/c_1|<1$ and $|q/c_2|<1$.
	\end{itemize}
	}
\end{cor}
\begin{proof}
	We assume $a_1\cdots a_{M+3}=q^2b_1\cdots b_{M+3}$.
	First, we prove (1) and (2).
	The integral
	\begin{align}
		\int_{q/a_{M+2}}^{q/a_{M+3}}\prod_{k=1}^{M+3}\frac{(a_kt)_\infty}{(b_kt)_\infty}d_qt,
	\end{align}
	is symmetric for $a_1,a_2,\ldots,a_{M+1}$ and for $b_1,b_2,\ldots,b_{M+3}$.
	We have
	\begin{align}
		&(\mbox{right-hand side of \eqref{inttoWpart}})=s(a_1,a_{M+1}).(\mbox{right-hand side of \eqref{inttoWpart}}),\label{proofsyma}\\
		&(\mbox{right-hand side of \eqref{inttoWpart}})=s(b_1,b_{M+3}).(\mbox{right-hand side of \eqref{inttoWpart}}),\label{proofsymb}
	\end{align}
	where $s(x,y):x\leftrightarrow y$.
	We put
	\begin{align}
		&a_i=x_i\quad (1\leq i\leq M),\ a_{M+1}=\frac{c_1c_2d}{A},\ a_{M+2}=\frac{c_1}{A},\ a_{M+3}=\frac{c_2}{A}\notag,\\
		&b_j=B_j\quad (1\leq j\leq M+2),\ b_{M+3}=\frac{c_1c_2}{Aq},\label{changeparam}
	\end{align}
	in \eqref{proofsyma} (resp. in \eqref{proofsymb}), and rewrite $A$, $B_j$ as $a$, $b_j$, we get \eqref{eqsyma} (resp. \eqref{eqsymb}).
	
	Next, we prove (3).
	By the definition of the Jackson integral, we have $\displaystyle \int_{q/a_{M+2}}^{q/a_{M+3}}+\int_{q/a_{M+3}}^{q/a_{M+1}}+\int_{q/a_{M+1}}^{q/a_{M+2}}=0$.
	Applying the transformation formula \eqref{inttoWpart} and putting  parameters as \eqref{changeparam}, we obtain \eqref{threeterm}.
\end{proof}
\begin{rem}
	
	The formula \eqref{eqsymb} is equivalent to Milne--Newcomb's transformation formula \cite[Theorem 5.5]{MN1996}.
\end{rem}
\begin{rem}
	
	We assume $a_1\cdots a_{M+3}=q^2b_1\cdots b_{M+3}$.
	We put
	\begin{align}
		\notag&W\left(\begin{array}{c}
			\{a_i\}_{1\leq i\leq M+3}\\
			\{b_i\}_{1\leq i\leq M+3}
		\end{array}\right)\\
		\notag&=
		\prod_{i=1}^M\frac{(qa_i/a_{M+2},qa_i/a_{M+3})_\infty}{(q^2a_ib_{M+3}/(a_{M+2}a_{M+3}))_\infty}\prod_{j=1}^{M+2}{(q^2b_jb_{M+3}/(a_{M+2}a_{M+3}))_\infty}\prod_{j=1}^{M+3}\frac{1}{(qb_j/a_{M+2},qb_j/a_{M+3})_\infty}\\
		\notag&\phantom{=}\times {(a_{M+1}/b_{M+3},a_{M+2}/a_{M+3},a_{M+3}/a_{M+2})_\infty}\frac{1}{a_{M+3}-a_{M+2}}\\
		&\phantom{=}\times W^{M,2}\left(\{a_i\}_{1\leq i\leq M};\frac{qb_{M+3}}{a_{M+2}a_{M+3}};\left\{\frac{1}{b_j}\right\}_{1\leq j\leq M+2};\frac{qb_{M+3}}{a_{M+2}},\frac{qb_{M+3}}{a_{M+3}};\frac{a_{M+1}}{b_{M+3}}\right).\label{defW}
	\end{align}
	Due to \eqref{inttoWpart}, we have
	\begin{align}
		W\left(\begin{array}{c}
			\{a_i\}_{1\leq i\leq M+3}\\
			\{b_i\}_{1\leq i\leq M+3}
		\end{array}\right)=\frac{1}{q(1-q)(q)_\infty}\int_{q/a_{M+3}}^{q/a_{M+2}}\prod_{k=1}^{M+3}\frac{(a_kt)_\infty}{(b_kt)_\infty}d_qt.\label{inttoW}
	\end{align}
	Hence, the function $W$ is symmetric for $a_1,a_2,\ldots,a_{M+1}$ and for $b_{1},b_2,\ldots,b_{M+3}$.
	Symmetries for $a_1,a_2,\ldots,a_M$ and for $b_1,b_2,\ldots,b_{M+2}$ are obvious.
	Only symmetries for $a_i\leftrightarrow a_{M+1}$ and for $b_i\leftrightarrow b_{M+3}$ are non-trivial.
	The formulas \eqref{eqsyma} and \eqref{eqsymb} describe these non-trivial symmetries of $i=1$ case.
\end{rem}
\begin{rem}\label{remM1trans}
	When $M=1$, the formulas \eqref{eqsyma} and \eqref{eqsymb} reduce to the well-known formula \cite[(2.10.1)]{GR}:
	\begin{align}
		{}_8W_7\left(a;b,c,d,e,f;\frac{a^2q^2}{bcdef}\right)=\frac{(a q,a q/(e f),\lambda q/e,\lambda q/f)_\infty}{(a q/e,a q/f,\lambda q,\lambda q/(e f))_\infty}{}_8W_7\left(\lambda;\frac{\lambda b}{a},\frac{\lambda c}{a},\frac{\lambda d}{a},e,f;\frac{aq}{ef}\right),
	\end{align}
	where $\lambda =q a^2/(b c d)$.
\end{rem}

\subsection{A $q$-difference system}
In this subsection, we construct a multivariable $q$-difference system associated with the following integral:
\begin{align}\label{intmulti}
	\varphi_{i,j}=\int_{q/a_i}^{q/a_j}\prod_{k=1}^{M+3}\frac{(a_k t)_{\infty}}{(b_k t)_{\infty}}d_q t,
\end{align}
where
\begin{align}
	a_1 \cdots a_{M+3}=q^2 b_1  \cdots b_{M+3}.\label{relationexponent}
\end{align}
This integral is a $q$-analog of
\begin{align}
	\int_C \prod_{i=1}^{M+3}(1-t_it)^{\nu_i}\ \ (\nu_1+\cdots +\nu_{M+3}=-2).
\end{align}
When $M=1$, this integral is a solution \eqref{RiePapint} of the Riemann--Papperitz equation.
Therefore we call \eqref{intmulti} the Jackson integral of Riemann--Papperitz type.

{First, we introduce the word ``rank''.
\begin{defn}
	Let $\{x_i\}_{1\leq i\leq m}$ be variables and $E=E(T_1,T_2,\ldots,T_k)$ be a homogeneous linear $q$-difference system, where $T_j$ is a $q$-difference operator with respect to some variables, i.e. $T_j=\prod_{i=1}^m T_{x_i}^{n_{i,j}}$.
	We put $F={\mathbb{C}(\{x_i\}_{1\leq i\leq m})}$.
	The rank of $E$ is $r$ if
	\begin{align}
		\dim_F((\mathrm{span}_F\{T_j^i y\mid 1\leq j\leq k,\ i\in\mathbb{Z}\})/{\sim})=r.
	\end{align}
	Here the relation $A(y)\sim B(y)$ means that $A(y)-B(y)=0$ holds for any solution $y$ of the system $E$. 
\end{defn}
\begin{lemm}\label{lemmarank}
	Let $E=E(T_1,T_2,\ldots,T_k)$ be a homogeneous linear $q$-difference system of rank $r$.
	If $y_1,\ldots,y_{r+1}$ are solutions of $E$, then $y_1,\ldots,y_{r+1}$ are linearly dependent over $K=\{C\mid T_j C=C\ (j=1,\ldots,k)\}$
\end{lemm}
\begin{proof}
	Suppose that $y$ is a solution of $E$.
	Since the rank of $E$ is $r$, $T_j^{r}y$ can be written as a linear combination of $y,T_jy,\ldots,T_j^{r-1}y$.
	Hence we have
	\begin{align}
		\det\begin{pmatrix}
			y_1&y_2&\cdots &y_{r+1}\\
			T_j y_1&T_j y_2&\cdots &T_j y_{r+1}\\
			\vdots&\vdots&&\vdots\\
			T_j^r y_1&T_j^r y_2&\cdots &T_j^r y_{r+1}
		\end{pmatrix}=0.
	\end{align}
	This completes the proof.
\end{proof}
}
In \cite{FN}, a $q$-difference equation associated with the Jackson  integral of Jordan--Pochhammer type:
\begin{align}
	\int_0^{\tau\infty}t^{\alpha-1} \frac{(Axt)_\infty}{(Bxt)_\infty}\prod_{i=2}^{M+3} \frac{(a_it)_\infty}{(b_it)_\infty}d_qt,\label{intJPgeneral}
\end{align}
is given.
{The integral \eqref{intJPgeneral} converges if $|q^\alpha|<1$ and $\left|q^{-\alpha}\dfrac{Aa_2a_3\cdots a_{M+3}}{Bb_2b_3\cdots b_{M+3}}\right|<1$.}
\begin{prop}[\cite{FN}]
	We assume $T_x\tau=q^l\tau$ $(l\in\mathbb{Z})$, {$|q^\alpha|<1$ and $\left|q^{-\alpha}\dfrac{Aa_2a_3\cdots a_{M+3}}{Bb_2b_3\cdots b_{M+3}}\right|<1$.}
	The integral \eqref{intJPgeneral} satisfies the following $q$-difference equation:
	\begin{align}
		\left[\sum_{k=0}^{M+2}(-1)^{k}x^{M+2-k}[e_k(a)T_x^{-1}-q^\alpha e_k(b)]\prod_{i=0}^{M+1-k}(B-Aq^iT_x)\prod_{i=0}^{k-1}(1-q^{-i}T_x)\right]y=0.\label{eqJPgeneral}
	\end{align}
	Here, $a=(a_2,\ldots,a_{M+3})$, $b=(b_2,\ldots,b_{M+3})$ and $e_k$ is the elementary symmetric function of degree $k$.
\end{prop}
\begin{rem}
	A linear $q$-difference system associated with the Jackson integral of Jordan--Pochhammer type \eqref{intJPgeneral} was obtained by Mimachi \cite{Mi1989} and Matsuo \cite{Ma}.
	A $q$-difference system for the Jackson integral of Selberg type, which contains the system for the Jordan--Pochhammer type, was obtained by Mimachi \cite{Mi1994}.
	In \cite{It}, a $q$-difference system for the Jackson integral of Selberg type was also discussed, and the above results  were summarized.
	Hence, for more details on the Jackson integral of Jordan--Pochhammer type, see \cite{It,Ma,Mi1989,Mi1994}.
\end{rem}
We put $\alpha=1$ and $A a_2\cdots a_{M+3}=q^2 Bb_2\cdots b_{M+3}$, then the equation \eqref{eqJPgeneral} is reducible:
\begin{align}
	\notag&\sum_{k=0}^{M+2}(-1)^{k}x^{M+2-k}[e_k(a)T_x^{-1}-q^\alpha e_k(b)]\prod_{i=0}^{M+1-k}(B-Aq^iT_x)\prod_{i=0}^{k-1}(1-q^{-i}T_x)\\
	&=(B-Aq^{-1}T_x)(1-q^{-1-M}T_x)E_M,
\end{align}
where the rank $M+1$ operator $E_M=E_M(T_x)$ is defined as
\begin{align}
	\notag E_M&=x^{M+2} T_x^{-1}\prod_{i=0}^{M}(B-Aq^iT_x)\\
	\notag&+\sum_{k=1}^{M+1}(-1)^{k}x^{M+2-k}[e_{k}(a)T_{x}^{-1}-qe_{k}(b)]\prod_{i=0}^{M-k}(B-Aq^iT_{x})\prod_{i=0}^{k-2}(1-q^{-i}T_{x})\\
	\label{defequationEM}&+(-1)^M \frac{a_2\cdots a_{M+3}}{B}T_x^{-1}\prod_{i=0}^{M}(1-q^{-i}T_x).
\end{align}
The Jackson integral
\begin{align}
	\int_0^{\tau\infty}\frac{(Axt)_\infty}{(Bxt)_\infty}\prod_{i=2}^{M+3}\frac{(a_it)_\infty}{(b_it)_\infty}d_qt,\ \ (Aa_2\cdots a_{M+3}=q^2 Bb_2\cdots b_{M+3}),
\end{align}
satisfies the equation $(B-Aq^{-1}T_x)(1-q^{-1-M}T_x)E_My=0$.
{If $B/A\neq q^M$, a general solution for $(B-Aq^{-1}T_x)(1-q^{-1-M}T_x) z=0$ is given by
\begin{align}
	C_1 \frac{\theta(Ax/q)}{\theta(Bx)}+C_2x^{M+1},
\end{align}
where $C_1$, $C_2$ are pseudo-constants, i.e. $T_xC_i=C_i$.}
Hence, we have
\begin{align}
	\label{pseudo}E_M \int_0^{\tau\infty}\frac{(Axt)_\infty}{(Bxt)_\infty}\prod_{i=2}^{M+3}\frac{(a_it)_\infty}{(b_it)_\infty}d_qt=C_1 \frac{\theta(Ax/q)}{\theta(Bx)} +C_2 x^{M+1},
\end{align}
where $C_1$, $C_2$ are pseudo-constants.
\begin{rem}
	If $\tau$ is independent from $x$,  the right-hand side can be written as
	\begin{align}
		C_1\frac{\theta(Ax\tau/q)}{\theta(Bx\tau)}+C_2 x^{M+1},
	\end{align}
	where $C_1$, $C_2$ are constants, not just pseudo-constants.
	This can be found as follows.
	
	Since $\tau$ is a constant, there exist pseudo-constants $C_1=C_1(x)$ and $C_2=C_2(x)$ such that
	\begin{align}
		E_M \int_0^{\tau\infty}\frac{(Axt)_\infty}{(Bxt)_\infty}\prod_{i=2}^{M+3}\frac{(a_it)_\infty}{(b_it)_\infty}d_qt=C_1 \frac{\theta(Ax\tau/q)}{\theta(Bx\tau)} +C_2 x^{M+1}.
	\end{align}
	On the other hand, the left-hand side of the equation \eqref{pseudo} has simple poles only at $Bx\tau\in q^{\mathbb{Z}}$.
	Thus the function
	\begin{align}
		\theta(Bx\tau)E_M \int_0^{\tau\infty}\frac{(Axt)_\infty}{(Bxt)_\infty}\prod_{i=2}^{M+3}\frac{(a_it)_\infty}{(b_it)_\infty}d_qt,
	\end{align}
	is holomorphic for $x\in\mathbb{C}^{\times}$.
	Therefore the function
	\begin{align}
		C_1\theta(Ax\tau/q)+C_2x^{M+1}\theta(Bx\tau),
	\end{align}
	is holomorphic for $x\in\mathbb{C}^{\times}$.
	For $x_0\in\mathbb{C}^{\times}$, we have
	\begin{align}
		\label{remprpseudo1}&\underset{x=x_0}{\mathrm{res}}C_1(x)\theta(Ax\tau/q)=-\underset{x=x_0}{\mathrm{res}}C_2(x) x^{M+1}\theta(Bx\tau),\\
		&\underset{x=x_0q}{\mathrm{res}}C_1(x)\theta(Ax\tau/q)=-\underset{x=x_0q}{\mathrm{res}}C_2(x) x^{M+1}\theta(Bx\tau)
	\end{align}
	Since $C_1$, $C_2$ are pseudo-constants and $\theta(qz)=(-1/z)\theta(z)$, we have
	\begin{align}
		&\underset{x=x_0q}{\mathrm{res}}C_1(x)\theta(Ax\tau/q)=-\frac{q}{Ax_0\tau}\underset{x=x_0}{\mathrm{res}}C_1(x)\theta(Ax\tau/q),\\
		\label{remprpseudo4}&\underset{x=x_0q}{\mathrm{res}}C_2(x)x^{M+1}\theta(Bx\tau)=-\frac{q^{M+1}}{Bx_0\tau}\underset{x=x_0}{\mathrm{res}}C_2(x)\theta(Bx\tau).
	\end{align}
	Due to \eqref{remprpseudo1}--\eqref{remprpseudo4} and the condition $B/A\neq q^M$, we have
	\begin{align}
		\underset{x=x_0}{\mathrm{res}}C_1(x)\theta(Ax\tau/q)=\underset{x=x_0}{\mathrm{res}}C_2(x) x^{M+1}\theta(Bx\tau)=0.
	\end{align}
	Therefore $C_1(x)$ (resp. $C_2(x)$) has simple poles at most $x=q^n/(A\tau)$ (resp. $x=q^n/(B\tau)$) ($n\in\mathbb{Z}$).
	We put $q=e^{2\pi \sqrt{-1}p}$ and $f_i(z)=C_i(e^{2\pi \sqrt{-1}z})$.
	Then $f_i$ is a elliptic function with periods $1$ and $p$.
	The order of $f_i$ is $0$ or $1$, however there is no elliptic function of order $1$.
	Therefore $f_i(z)=C_i(e^{2\pi \sqrt{-1}z})$ is a constant.
	
\end{rem}
The following lemma is useful to calculate these pseudo-constants.
\begin{lemm}\label{lemmarespsi}
	Suppose $|d_1\cdots d_{M+2}/(c_1\cdots c_{M+2})|<1$.
	We have
	\begin{align}
		\lim_{z\to1} (1-z){}_{M+2}\psi_{M+2}\left(\begin{array}{c}
			c_1,\ldots,c_{M+2}\\
			d_1,\ldots,d_{M+2}
		\end{array};z\right)=\frac{(c_1,\ldots,c_{M+2})_\infty}{(d_1,\ldots,d_{M+2})_\infty}.\label{respsi}
	\end{align}
	Here, ${}_{M+2}\psi_{M+2}$ is the bilateral $q$-hypergeometric series:
	\begin{align}
		{}_{M+2}\psi_{M+2}\left(\begin{array}{c}
			c_1,\ldots,c_{M+2}\\
			d_1,\ldots,d_{M+2}
		\end{array};z\right)=\sum_{l\in\mathbb{Z}}\frac{(c_1,\ldots,c_{M+2})_l}{(d_1,\ldots,d_{M+2})_l}z^l\ \ (|d_1\cdots d_{M+2}/(c_1\cdots c_{M+2})| <|z|<1).
	\end{align}
\end{lemm}
\begin{proof}
	The negative power part of the function ${}_{M+2}\psi_{M+2}$ is holomorphic at $z=1$ because it can be rewritten as
	\begin{align}
		\sum_{l\in\mathbb{Z}_{<0}}\frac{(c_1,\ldots,c_{M+2})_l}{(d_1,\ldots,d_{M+2})_l}z^l=\sum_{l\in\mathbb{Z}_{>0}}\frac{(q/d_1,\ldots,q/d_{M+2})_l}{(q/c_1,\ldots,q/c_{M+2})_l}\left(\frac{d_1\cdots d_{M+2}}{c_1\cdots c_{M+2}}z\right)^l.
	\end{align} 
	Thus we have
	\begin{align}
		\lim_{z\to1}(1-z){}_{M+2}\psi_{M+2}\left(
		\begin{array}{c}
			c_1,\ldots,c_{M+2}\\
			d_1,\ldots,d_{M+2}
		\end{array}
		;z\right)=\lim_{z\to1}(1-z)\sum_{l=0}^\infty\frac{(c_{1},\ldots,c_{M+2})_{l}}{(d_{1},\ldots,d_{M+2})_{l}}z^{l}.
	\end{align}
	We can check that the series
	\begin{align}
		\sum_{l=0}^\infty\left(\frac{(c_{1},\ldots,c_{M+2})_{l}}{(d_{1},\ldots,d_{M+2})_{l}}-\frac{(c_{1},\ldots,c_{M+2})_{\infty}}{(d_{1},\ldots,d_{M+2})_{\infty}}\right)z^{l},
	\end{align}
	converges in $|z|<|q^{-1}|$.
	Hence, we have
	\begin{align}
		\notag&\lim_{z\to1}(1-z)\sum_{l=0}^\infty\frac{(c_{1},\ldots,c_{M+2})_{l}}{(d_{1},\ldots,d_{M+2})_{l}}z^{l}\\
		\notag&=\frac{(c_{1},\ldots,c_{M+2})_{\infty}}{(d_{1},\ldots,d_{M+2})_{\infty}}\lim_{z\to1}(1-z)\sum_{l=0}^\infty z^{l}+\lim_{z\to1}(1-z)\sum_{l=0}^\infty\left(\frac{(c_{1},\ldots,c_{M+2})_{l}}{(d_{1},\ldots,d_{M+2})_{l}}-\frac{(c_{1},\ldots,c_{M+2})_{\infty}}{(d_{1},\ldots,d_{M+2})_{\infty}}\right)z^{l}\\
		&=\frac{(c_{1},\ldots,c_{M+2})_{\infty}}{(d_{1},\ldots,d_{M+2})_{\infty}}.
	\end{align}
	Therefore we get the desired equation \eqref{respsi}.
\end{proof}
\begin{prop}\label{propintsolEM}
	Suppose $\tau\in\{q/a_{2},\ldots,q/a_{M+3},q/(Ax)\}$, $B/A\notin q^{\mathbb{Z}}$ {and $A a_2\cdots a_{M+3}=q^2 Bb_2\cdots b_{M+3}$}.
	We have
	\begin{align}
		E_M \int_0^{\tau\infty}\frac{(Axt)_\infty}{(Bxt)_\infty}\prod_{i=2}^{M+3}\frac{(a_it)_\infty}{(b_it)_\infty}d_qt=-\prod_{i=0}^{M-1}(B-Aq^i)\cdot q(1-q)x^{M+1}.
	\end{align}
	In particular, we have
	\begin{align}
		E_M \int_{\tau_1}^{\tau_2}\frac{(Axt)_\infty}{(Bxt)_\infty}\prod_{i=2}^{M+3}\frac{(a_it)_\infty}{(b_it)_\infty}d_qt=0,
	\end{align}
	where $\tau_1,\tau_2\in\{q/a_{2},\ldots,q/a_{M+3},q/(Ax)\}$.
\end{prop}
\begin{proof}
	If $\tau=q/a_k$, the integral
	\begin{align}
		\varphi_\tau(x)&=\int_0^{\tau\infty}\frac{(Axt)_\infty}{(Bxt)_\infty}\prod_{i=2}^{M+3}\frac{(a_it)_\infty}{(b_it)_\infty}d_qt\\
		&=\int_0^{\tau}\frac{(Axt)_\infty}{(Bxt)_\infty}\prod_{i=2}^{M+3}\frac{(a_it)_\infty}{(b_it)_\infty}d_qt,
	\end{align}
	is holomorphic at $x=0$.
	Also if $\tau=q/(Ax)$, this integral is holomorphic at $x=\infty$.
	Hence, we have \begin{align}
		E_M\varphi_\tau=C x^{M+1}\label{intremainder}
	\end{align} where $C$ is a constant.
	
	For $\tau=q/a_k$, we have
	\begin{align}
		\notag&\int_0^{\tau}\frac{(Axt)_\infty}{(Bxt)_\infty}\prod_{i=2}^{M+3}\frac{(a_it)_\infty}{(b_it)_\infty}d_qt\\
		\notag&=\sum_{l=0}^\infty \frac{(A/B)_l}{(q)_l}x^l \int_0^{\tau}t^l \prod_{i=2}^{M+3}\frac{(a_it)_\infty}{(b_it)_\infty}d_qt\\
		&=\sum_{l=0}^\infty \frac{(A/B)_l}{(q)_l}x^l(1-q)\tau^{l+1}\prod_{i=2}^{M+3}\frac{(a_i\tau)_\infty}{(b_i\tau)_\infty}{}_{M+2}\psi_{M+2}\left(\begin{array}{c}
			b_2\tau,\ldots,b_{M+3}\tau\\
			a_2\tau,\ldots,a_{M+3}\tau
		\end{array};q^{l+1}\right).
	\end{align}
	due to the $q$-binomial theorem.
	Note that the negative power part of this ${}_{M+2}\psi_{M+2}$ {vanishes} since $\tau=q/a_k$, i.e, this ${}_{M+2}\psi_{M+2}$ can be written by ${}_{M+2}\varphi_{M+1}$.
	From direct calculations, the coefficient $C$ of $x^{M+1}$ in $E_M\varphi_\tau$ is given as
	\begin{align}
		\notag C&=\prod_{i=0}^{M-1}(B-Aq^i)\cdot q(1-q)\prod_{i=2}^{M+3}\frac{(a_i\tau)_\infty}{(b_i\tau)_\infty}\\
		&\times\sum_{l=1}^
		{M+2} (-1)^l[e_l (a) q^{-l}- e_l(b)]\tau^l{}_{M+2}\psi_{M+2}\left(\begin{array}{c}
			b_2\tau,\ldots,b_{M+3}\tau\\
			a_2\tau,\ldots,a_{M+3}\tau
		\end{array};q^l\right).
	\end{align}
	The function ${}_{M+2}\psi_{M+2}$ satisfies the following $q$-difference equation:
	\begin{align}
		\left[\prod_{i=2}^{M+3}(1-a_i\tau q^{-1} T_z)-z\prod_{i=2}^{M+3}(1-b_i\tau T_z)\right]{}_{M+2}\psi_{M+2}\left(\begin{array}{c}
			b_2\tau,\ldots,b_{M+3}\tau\\
			a_2\tau,\ldots,a_{M+3}\tau
		\end{array};z\right)=0.
	\end{align}
	Expanding this equation, we have
	\begin{align}
		\sum_{l=0}^
		{M+2} (-1)^l[e_l (a) q^{-l}- e_l(b)]\tau^l{}_{M+2}\psi_{M+2}\left(\begin{array}{c}
			b_2\tau,\ldots,b_{M+3}\tau\\
			a_2\tau,\ldots,a_{M+3}\tau
		\end{array};q^l z\right)=0.
	\end{align}
	Considering the limit $z\to1$, we finally obtain
	\begin{align}
		\notag C&=-\prod_{i=0}^{M-1}(B-Aq^i)\cdot q(1-q)\prod_{i=2}^{M+3}\frac{(a_i\tau)_\infty}{(b_i\tau)_\infty}\lim_{z\to1}(1-z){}_{M+2}\psi_{M+2}\left(\begin{array}{c}
			b_2\tau,\ldots,b_{M+3}\tau\\
			a_2\tau,\ldots,a_{M+3}\tau
		\end{array};z\right)\\
		&=-\prod_{i=0}^{M-1}(B-Aq^i)\cdot q(1-q).
	\end{align}

	For $\tau=q/(Ax)$, we have
	\begin{align}
		\int_0^\tau\frac{(Axt)_\infty}{(Bxt)_\infty}\prod_{i=2}^{M+3}\frac{(a_it)_\infty}{(b_it)_\infty}d_qt&=\frac{q}{Ax}\int_0^1 \frac{(qt)_\infty}{(qBt/A)_\infty}\prod_{i=2}^{M+3}\frac{(qa_it/(Ax))_\infty}{(qb_it/(Ax))_\infty}d_qt\notag\\
		&=\frac{q}{Ax}\left(\int_0^1\frac{(qt)_\infty}{(qBt/A)_\infty}d_qt+O\left(\frac{1}{x}\right)\right).
	\end{align}
	The last integral is factorized by the $q$-binomial theorem as follows:
	\begin{align}
		\int_0^{1}\frac{(qt)_\infty}{(qBt/A)_\infty}d_q t= (1-q) \frac{(q)_\infty}{(qB/A)_\infty}\frac{(q^2 B/A)_\infty}{(q)_\infty}=\frac{(1-q)q}{1-qB/A}.
	\end{align}
	The degree of the operator $E_M$ in $x$ is $M+2$, thus we have
	\begin{align}
		E_M\varphi_\tau=\frac{q}{A}\frac{(1-q)q}{1-qB/A}q \prod_{i=0}^{M}(B-Aq^iq^{-1})x^{M+1}+O(x^{M}).
	\end{align}
	By the equation \eqref{intremainder}, we get
	\begin{align}
		E_M\varphi_\tau=-\prod_{i=0}^{M-1}(B-Aq^i)\cdot q(1-q)x^{M+1}.
	\end{align}
	This completes the proof.
\end{proof}
\begin{rem}\label{remqRPM1}
	When $M=1$, the equation $E_1 y=0$ is equivalent to the variant of $q$-hypergeometric equation of degree three \cite{HMST}.
	%
	Lemma \ref{lemmarespsi} and Proposition \ref{propintsolEM} for $M=1$ were obtained in \cite{FN}, and integral solutions for the variant were given.
	See also \cite{AT} for the integral solutions.
\end{rem}
{We put $a_1=Ax$, $b_1=Bx$, and for $2\leq i\leq M+3$, $a_i$, $b_i$ are independent from $x$.
Then for any function $f(\{a_i\}_{1\leq i\leq M+3};\{b_i\}_{1\leq i\leq M+3})$, we have
\begin{align}
	T_x f(\{a_i\};\{b_i\})=f(q a_1,\{a_i\}_{2\leq i\leq M+3};qb_1,\{b_i\}_{2\leq i\leq M+3})=T_{a_1}T_{b_1}f(\{a_i\};\{b_i\}).
\end{align}
Therefore we have $T_x=T_{a_1}T_{b_1}$ as an operator.}
We have the following proposition.
\begin{prop}\label{propmultiord}
	{Suppose $a_1\cdots a_{M+3}=q^2b_1\cdots b_{M+3}$.}
	The Jackson integral of the Riemann--Papperitz type $\varphi_{i,j}$ \eqref{intmulti} satisfies the rank $M+1$ equation $\hat{E}_{M}\varphi_{i,j}=0$.
	Here the operator $\hat{E}_{M}=\hat{E}_M(T_{a_1}T_{b_1})$ is given by
	\begin{align}
		\hat{E}_M\notag&=b_{1}^{M+2}T^{-1}\prod_{i=0}^{M}(1-(a_{1}q^i/b_{1})T)\\
		\notag&+\sum_{k=1}^{M+1}(-1)^k b_{1}^{M+2-k}[e_k(\hat{a})T^{-1}-q e_k(\hat{b})]\prod_{i=0}^{M-k}(1-(a_{1}q^i/b_{1})T)\prod_{i=0}^{k-2}(1-q^{-i}T)\\
		&+(-1)^{M}a_2\cdots a_{M+3}T^{-1}\prod_{i=0}^M(1-q^{-i}T),
	\end{align}
	where $T=T_{a_{1}}T_{b_{1}}$ and $\hat{a}=(a_2,\ldots,a_{M+3})$, $\hat{b}=(b_2,\ldots,b_{M+3})$.
\end{prop}
\begin{proof}
	By rewriting the operator $E_M$ \eqref{defequationEM} in $a_1$, $b_1$, we have the desired equation.
\end{proof}
\begin{rem}
	The operator $T=T_{a_{1}}T_{b_{1}}$ preserves the condition $a_1\cdots a_{M+3}=q^2 b_1\cdots b_{M+3}$ \eqref{relationexponent}.
	Similarly, we can only consider $T_{a_k}T_{a_l}^{-1}$, $T_{b_k}T_{b_l}^{-1}$, $T_{a_k}T_{b_l}$, etc. as $q$-shift operators under that condition.
\end{rem}
%
Proposition \ref{propmultiord} means that the integral $\varphi_{i,j}$ satisfies a $q$-difference equation of rank $M+1$.
Proposition \eqref{propmulti3term} given below leads to a multivariable $q$-difference system of rank $M+1$ satisfied by the integral.
\begin{prop}\label{propmulti3term}
	We suppose $T_{a_i}\tau,T_{b_i}\tau\in\tau q^\mathbb{Z}=\{\tau q^n\mid n\in\mathbb{Z}\}$.
	{We put
	\begin{align}\label{integrandpsi}
		&\psi=\prod_{i=1}^{M+3}\frac{(a_i t)_{\infty}}{(b_i t)_{\infty}}.
\end{align}}
	The integral
	\begin{align}
		&y=\int_0^{\tau\infty}\psi d_qt,\label{intpropint}\\
	\end{align}
	satisfies the following three term relations:
	\begin{align}
		\label{int3term1}&[(a_{1}-a_k q)T_{a_k}T_{a_l}^{-1}-(a_l-a_k q)T_{a_k}T_{a_{1}}^{-1}+(a_l-a_{1})]y=0,\\
		&[(b_{1}-a_{1})T_{a_{1}}T_{a_k}^{-1}-(q^{-1}a_k-a_{1})T_{a_{1}}T_{b_{1}}+(q^{-1}a_k-b_{1})]y=0,\\
		&[(b_{1}-q^{-1}b_l)T_{b_k}T_{b_l}^{-1}-(b_k-q^{-1}b_l)T_{b_{1}}T_{b_l}^{-1}+(b_k-b_{1})]y=0,\\
		&[(a_{1}-b_{1})T_{b_k}T_{b_{1}}^{-1}-(b_k q-b_{1})(T_{a_{1}}T_{b_{1}})^{-1}+(b_k q-a_{1})]y=0,\\
		&[(b_{1}-a_k)T_{a_k}T_{b_l}-(b_l-a_k)T_{a_k}T_{b_{1}}+(b_l-b_{1})]{y}=0,\\
		&[(b_l-a_1)T_{a_k}^{-1}T_{b_l}^{-1}-(b_l-a_k)T_{a_1}^{-1}T_{b_l}^{-1}+(a_1-a_k)]y=0,
	\end{align}
	where $2\leq k\neq l\leq M+3$.
	Also the integral \eqref{intpropint} satisfies
	\begin{align}
		&\left(q\prod_{k=1}^{M+3}T_{a_k}T_{b_k}-1\right)y=0.
	\end{align}
\end{prop}
\begin{proof}
	We only prove the equation \eqref{int3term1} for the three term relations.
	Other three term relations can be proved in the same way.
	The integrand $\psi$ given in \eqref{integrandpsi} satisfies 
	\begin{align}
		[(a_{1}-a_k q)T_{a_k}T_{a_l}^{-1}-(a_l-a_k q)T_{a_k}T_{a_{1}}^{-1}+(a_l-a_{1})]\psi=0.
	\end{align}
	This equation is derived by the relation for rational functions of $t$:
	\begin{align}
		(a_{1}-a_k q)\frac{1-a_l q^{-1}t}{1-a_k t}-(a_l-a_k q)\frac{1-a_{1} q^{-1}t}{1-a_kt}+(a_l-a_{1})=0.
	\end{align}
	By integrating in $t$, we have the relation \eqref{int3term1}.
	Also the integrand $\psi$ satisfies
	\begin{align}
		T_t\psi=\left(\prod_{k=1}^{M+3}T_{a_k}T_{b_k}\right)\psi,
	\end{align}
	hence we have
	\begin{align}
		q^{-1}y=\left(\prod_{k=1}^{M+3}T_{a_k}T_{b_k}\right)y.
	\end{align}
	Here we used the following type formula \cite{FN}:
	\begin{align}
		\int_0^{\tau\infty}T_t^i T_x^j f(t,x)d_qt=q^{-i}T_x^j \int_0^{\tau\infty}f(t,x)d_qt,
	\end{align}
	if $T_x\tau=q^l\tau$ $(l\in\mathbb{Z})$.
	This can be checked directly:
	\begin{align}
		\int_0^{\tau\infty}T_t^i T_x^j f(t,x)d_qt&=(1-q)\sum_{n\in\mathbb{Z}}f(q^i \tau q^n,T_x^j x)\tau q^n\\
		&=(1-q)\sum_{n\in\mathbb{Z}}f(q^{i+n-jl} T_x^j \tau,T_x^j x)(T_x^j\tau) q^{i+n-jl}q^{-i}\\
		&=q^{-i} T_x^j(1-q)\sum_{m\in\mathbb{Z}}f(q^m \tau,x)\tau q^m\\
		&=q^{-i}T_x^j \int_0^{\tau\infty}f(t,x)d_qt.
	\end{align}
\end{proof}


In conclusion, the integral $\varphi_{i,j}$ \eqref{intmulti} satisfies the following $q$-difference system.
\begin{thm}\label{thmmultisystem}
	Suppose $a_1\cdots a_{M+3}=q^2b_1\cdots b_{M+3}$.
	The Jackson integral of Riemann--Papperitz type $\varphi_{i,j}$ \eqref{intmulti} satisfies the following system:
	\begin{align}
		&\hat{E}_M y=0,\\
		\label{sys3term1}&[(a_{1}-a_k q)T_{a_k}T_{a_l}^{-1}-(a_l-a_k q)T_{a_k}T_{a_{1}}^{-1}+(a_l-a_{1})]y=0,\\
		&[(b_{1}-a_{1})T_{a_{1}}T_{a_k}^{-1}-(q^{-1}a_k-a_{1})T_{a_{1}}T_{b_{1}}+(q^{-1}a_k-b_{1})]y=0,\\
		&[(b_{1}-q^{-1}b_l)T_{b_k}T_{b_l}^{-1}-(b_k-q^{-1}b_l)T_{b_{1}}T_{b_l}^{-1}+(b_k-b_{1})]y=0,\\
		&[(a_{1}-b_{1})T_{b_k}T_{b_{1}}^{-1}-(b_k q-b_{1})(T_{a_{1}}T_{b_{1}})^{-1}+(b_k q-a_{1})]y=0,\\
		&[(b_{1}-a_k)T_{a_k}T_{b_l}-(b_l-a_k)T_{a_k}T_{b_{1}}+(b_l-b_{1})]y=0,\\
		\label{sys3term5}&[(b_l-a_1)T_{a_k}^{-1}T_{b_l}^{-1}-(b_l-a_k)T_{a_1}^{-1}T_{b_l}^{-1}+(a_1-a_k)]y=0,
		\\
		&\left(q\prod_{k=1}^{M+3}T_{a_k}T_{b_k}-1\right)y=0,
	\end{align}
	where $2\leq k\neq l\leq M+3$.
	We call this system $q$-$RP^M$.
	The rank of $q$-$RP^M$ is $M+1$.
\end{thm}
\begin{proof}
	We already proved that the integral $\varphi_{i,j}$ satisfies the system $q$-$RP^M$.
	We will prove that the rank of $q$-$RP^M$ is $M+1$.
	Using the three term relations \eqref{sys3term1}--\eqref{sys3term5}, the functions $T_{a_k}T_{a_l}^{-1}y$, $T_{b_k}T_{b_l}^{-1}y$, $T_{a_k}T_{b_l}y$ can be rewritten as linear combinations of $y$, $(T_{a_{1}}T_{b_{1}})^{\pm1}y$, $(T_{a_{1}}T_{b_{1}})^{\pm2}y$, ...
	The operator $\hat{E}_M$ is an ordinary $q$-difference operator of rank $M+1$, therefore the rank of $q$-$RP^M$ is $M+1$.
\end{proof}
\begin{rem}\label{remqRPM2}
	If $M=1$, the equation $E_1y=0$ can be regarded as a $q$-analog of the Riemann--Papperitz system (see \cite{FN}).
	Here the Riemann--Papperitz equation is a second order Fuchsian differential equation with three singularities $\{t_1,t_2,t_3\}$ on the Riemann sphere $\mathbb{P}^1=\mathbb{C}\cup\{\infty\}$.
	The above system is an extension of the $q$-Riemann--Papperitz system, hence we call it $q$-$RP^M$.
\end{rem}
As a corollary of Theorem \ref{thmkajiint} and \ref{thmmultisystem}, we have the following:
\begin{cor}\label{corKajiharaeq}
	The function $W\left(\begin{array}{c}
		\{a_i\}_{1\leq i\leq M+3}\\
		\{b_i\}_{1\leq i\leq M+3}
	\end{array}\right)$ given in \eqref{defW} satisfies the system $q$-$RP^M$.
\end{cor}
\begin{proof}
	The function $W\left(\begin{array}{c}
		\{a_i\}_{1\leq i\leq M+3}\\
		\{b_i\}_{1\leq i\leq M+3}
	\end{array}\right)$ has the integral representation \eqref{inttoW}.
\end{proof}

Next, we give a basis of the space of solutions for the system $q$-$RP^M$,
\begin{thm}\label{thmindependence}
	{For generic $\{a_i\}_{1\leq i\leq M+3}$ and $\{b_i\}_{1\leq i\leq M+3}$ with $a_1\cdots a_{M+3}=q^2b_1\cdots b_{M+3}$,} the integrals $\varphi_{2,M+3},\varphi_{3,M+3},\ldots,\varphi_{M+2,M+3}$ are basis of the space of solutions for $q$-$RP^{M}$ over $K=\{C\mid T_{a_k}T_{a_l}^{-1}C=T_{b_k}T_{b_l}^{-1}C=T_{a_k}T_{b_l}C=C\}$.
%
\end{thm}
\begin{proof}
%
	Due to Lemma \ref{lemmarank}, it is enough to check that $\varphi_{2,M+3},\varphi_{3,M+3},\ldots,\varphi_{M+2,M+3}$ are linearly independent on $K$.
	It is sufficient to prove it in a special case.
	We  put $a_i=b_i$ ($i=2,\ldots,M+3$).
	Also we put $b_1=x$, then we have $a_1=q^2x$.
	By a simple calculation, we have
	\begin{align}
		\varphi_{i,M+3}=\int_{q/a_i}^{q/a_{M+3}}\frac{1}{(1-xt)(1-qxt)}d_q t=\frac{(1-q)q(a_i-a_{M+3})}{(a_i-xq)(a_{M+3}-xq)}.
	\end{align}
	Suppose
	\begin{align}
		C_2 \varphi_{2,M+3}+\cdots+C_{M+2}\varphi_{M+2,M+3}=0,
	\end{align}
	where $C_i\in K$.
	Then every pseudo{-}constant $C_i$ satisfies $T_xC_i=C_i$.
	We have
	\begin{align}
		(C_2,\ldots,C_{M+2})((p_2,\ldots,p_{M+2})^\mathrm{T})=0,
	\end{align}
	where ${v}^{\mathrm{T}}$ is the transpose of $v$ and
	\begin{align}
		p_i=(a_i-a_{M+3})\prod_{\substack{2\leq j\leq M+2\\j\neq i}}(a_j-xq).
	\end{align}
	By the assumption $a_i\neq a_j$, the polynomials $p_i$ of degree $M$ are basis of $\{f(x)\in \mathbb{C}[x]\mid\deg f(x)\leq M\}$.
	Hence there exists a constant matrix $G\in GL(M+1,\mathbb{C})$ such that
	\begin{align}
		(p_2,\ldots,p_{M+2})^\mathrm{T}=G (1,x,\ldots,x^{M})^{\mathrm{T}}.
	\end{align}
	We put $(\tilde{C}_0,\ldots,\tilde{C}_{M})=(C_2,\ldots,C_{M+2})G$, then $\tilde{C}_0,\ldots,\tilde{C}_{M}$ are pseudo-constants for $x$ and 
	\begin{align}
		(\tilde{C}_0,\ldots,\tilde{C}_{M})((1,x,\ldots,x^{M})^{\mathrm{T}})=0.
	\end{align}
	Applying $T_x^k$ ($k=1,\ldots,M$), we have
	\begin{align}
		(\tilde{C}_0,\ldots,\tilde{C}_{M})\begin{pmatrix}
			1&1&\cdots& 1\\
			x&qx&\cdots&q^{M}x\\
			\vdots&\vdots & &\vdots\\
			x^{M}&(qx)^{M}&\cdots&(q^{M}x)^{M}
		\end{pmatrix}=(0,\ldots,0).
	\end{align}
	By the Vandermonde determinant, we have $(\tilde{C}_0,\ldots,\tilde{C}_{M})=(0,\ldots,0)$.
	Since $G$ is invertible, we finally obtain $(C_2,\ldots,C_{M+2})=(0,\ldots,0)$.
\end{proof}

\section{Degeneration to the $q$-Appell--Lauricella system}\label{secdegene}
In this section, we consider a degeneration for the system $q$-$RP^M$.
We show that the degeneration of $q$-$RP^M$ includes the $q$-Appell--Lauricella system \eqref{qALsyscap1}, \eqref{qALsyscap2}.
By taking the same degeneration for Kajihara's $q$-hypergeometric series $W^{M,2}$, we construct some solutions for the $q$-Appell--Lauricella system.
It is well known that the $q$-Appell--Lauricella function
\begin{align}
	\varphi_D\left(\begin{array}{c}
		A;\{B_i\}_{1\leq i\leq M}\\
		C
	\end{array};\{x_i\}_{1\leq i\leq M}\right)=\sum_{l\in(\mathbb{Z}_{\geq0})^M}\frac{(A)_{|l|}}{(C)_{|l|}}\prod_{i=1}^{M}\left(\frac{(B_i)_{l_i}}{(q)_{l_i}}x_i^{l_i}\right)\quad (|x_i|<1),
\end{align}
is a solution of the $q$-Appell--Lauricella system {(cf. \cite[(59)]{Noumi})}.
{We give transformation formulas between $\varphi_D$ and our solutions.}


First, we put $b_{M+3}=q^\lambda a_{M+3}=q^{-n}$ ($n\in\mathbb{Z})$, and consider the limit $a_{M+3}\to\infty$, i.e. $n\to+\infty$.
In the following, we suppose $|q^{\lambda+1}|<1$.
We define a pseudo{-}constant
\begin{align}\label{defC0}
	C_0(t)=t^\lambda\frac{\theta(b_{M+3}t)}{\theta(a_{M+3}t)}.
\end{align}
For $1\leq j\leq M+2$, we have
\begin{align}
	\notag&\lim_{a_{M+3}\to\infty} C_0\left(\frac{q}{a_j}\right)\int_0^{q/a_j}\prod_{i=1}^{M+3}\frac{(a_i t)_\infty}{(b_i t)_\infty}d_qt=\lim_{a_{M+3}\to\infty}\int_0^{q/a_j}C_0(t)\prod_{i=1}^{M+3}\frac{(a_i t)_\infty}{(b_i t)_\infty}d_qt\\
	&=\int_0^{q/a_j}t^\lambda\prod_{i=1}^{M+2}\frac{(a_i t)_\infty}{(b_i t)_\infty}d_qt.\label{just4}
\end{align}
{Here we exchange the limit and integration.
This is justified in Appendix \ref{appA}.}
On the other hand, we have
\begin{align}
	\notag&\int_0^{q/a_{M+3}}\prod_{i=1}^{M+3}\frac{(a_i t)_\infty}{(b_i t)_\infty}d_qt=\frac{q}{a_{M+3}}\int_0^{1}\prod_{i=1}^{M+3}\frac{(a_i qt/a_{M+3})_\infty}{(b_i qt/a_{M+3})_\infty}d_qt\\
	&=\frac{q}{a_{M+3}}\left(\int_0^1\frac{(qt)_\infty}{(q^{\lambda+1}t)_\infty}d_qt+O\left(\frac{1}{a_{M+3}}\right)\right).
\end{align}
According to the condition $|q^{\lambda+1}|<1$, we have
\begin{align}
	\frac{1}{a_{M+3}}C_0\left(\frac{q}{a_j}\right)=\frac{1}{q^{-n}q^{-\lambda}}\left(\frac{q}{a_j}\right)^\lambda\frac{\theta(q^{-n} q/a_j)}{\theta(q^{-n}q^{-\lambda}q/a_j)}=(q^{1+\lambda})^{n}{\left(\frac{q^2}{a_j}\right)^\lambda}\frac{\theta(q/a_j)}{\theta(q^{-\lambda}q/a_j)}\xrightarrow{n\to+\infty}0.
\end{align}
Here we used $\theta(x q^{-1})=-xq^{-1}\theta(x)$.
Hence, we have
\begin{align}
	\lim_{n\to+\infty} C_0\left(\frac{q}{a_j}\right)\int_{q/a_{M+3}}^{q/a_j}\prod_{i=1}^{M+3}\frac{(a_i t)_\infty}{(b_i t)_\infty}d_qt=\int_0^{q/a_j}t^\lambda\prod_{i=1}^{M+2}\frac{(a_i t)_\infty}{(b_i t)_\infty}d_qt.\label{reductionintegral1}
\end{align}
Next we consider the limit $a_{M+2}\to0$ with $b_{M+2}=q^\beta a_{M+2}$. Here $a_1\cdots a_{M+1}=q^{\lambda+\beta+2} b_1\cdots b_{M+1}$.
For $1\leq j\leq M+1$, we easily find
\begin{align}
	\lim_{a_{M+2}\to0}\int_0^{q/a_j}t^\lambda\prod_{i=1}^{M+2}\frac{(a_i t)_\infty}{(b_i t)_\infty}d_qt=\int_0^{q/a_j}t^\lambda\prod_{i=1}^{M+1}\frac{(a_i t)_\infty}{(b_i t)_\infty}d_qt.\label{just5}
\end{align}
{This limit can be checked by Tannery's theorem (see Appendix \ref{appA}).}

We consider the same degenerations for the system $q$-$RP^M$.
First we put $b_{M+3}=q^\lambda a_{M+3}$.
We put $z=C_0(\tau)y$ ($\tau=q/a_1,\ldots,q/a_{M+2}$), and we should consider the degeneration $a_{M+3}\to\infty$ for the system associated with $z$.
The function $C_0(\tau)$, given by \eqref{defC0}, is a pseudo-constant for $a_i$, $b_i$ ($1\leq i\leq M+2$), and 
\begin{align}
	C_0(\tau)T_{a_{M+3}}T_{b_{M+3}}y=q^\lambda T_{a_{M+3}}T_{b_{M+3}}C_0(\tau)y.
\end{align}
Hence, the function $z$ satisfies
\begin{align}
	&\hat{E}_M z=0,\\
	&[(a_{1}-a_k q)T_{a_k}T_{a_l}^{-1}-(a_l-a_k q)T_{a_k}T_{a_{1}}^{-1}+(a_l-a_{1})]z=0,\\
	&[(b_{1}-a_{1})T_{a_{1}}T_{a_k}^{-1}-(q^{-1}a_k-a_{1})T_{a_{1}}T_{b_{1}}+(q^{-1}a_k-b_{1})]=0,\\
	&[(b_{1}-q^{-1}b_l)T_{b_k}T_{b_l}^{-1}-(b_k-q^{-1}b_l)T_{b_{1}}T_{b_l}^{-1}+(b_k-b_{1})]z=0,\\
	&[(a_{1}-b_{1})T_{b_k}T_{b_{1}}^{-1}-(b_k q-b_{1})(T_{a_{1}}T_{b_{1}})^{-1}+(b_k q-a_{1})]z=0,\\
	&[(b_{1}-a_k)T_{a_k}T_{b_l}-(b_l-a_k)T_{a_k}T_{b_{1}}+(b_l-b_{1})]z=0,\\
	&[(b_l-a_1)T_{a_k}^{-1}T_{b_l}^{-1}-(b_l-a_k)T_{a_1}^{-1}T_{b_l}^{-1}+(a_1-a_k)]z=0,
	\\
	&\left(q^{\lambda+1}\prod_{k=1}^{M+3}T_{a_k}T_{b_k}-1\right)z=0,
\end{align}
for $2\leq k\neq l\leq M+2$.
By the definition of $\hat{E}_M$, we have
\begin{align}
	\lim_{a_{M+3}\to\infty}\frac{1}{a_{M+3}}\hat{E}_M z=\hat{E}_M'z,
\end{align}
where $\hat{a}'=(a_2,\ldots,a_{M+2})$ $\hat{b}'=(b_2,\ldots,b_{M+2})$, $T=T_{a_1}T_{b_1}$ and
\begin{align}
	\notag\hat{E}_M'&=\sum_{k=1}^{M+1}(-1)^k b_{1}^{M+2-k}[e_{k-1}(\hat{a}')T^{-1}- q^{\lambda+1} e_{k-1}(\hat{b}')]\prod_{n=0}^{M-k}(1-(a_{1}q^n/b_{1})T)\prod_{n=0}^{k-2}(1-q^{-n}T)\\
	&+(-1)^{M}a_2\cdots a_{M+2}T^{-1}\prod_{n=0}^M(1-q^{-n}T).
\end{align}
Note that
\begin{align}
	e_{k}(\hat{a})=e_{k}(\hat{a}')+a_{M+3}e_{k-1}(\hat{a}').
\end{align}
Next we put $b_{M+2}=q^\beta a_{M+2}$, then we easily find
\begin{align}
	\lim_{a_{M+2}\to0}\hat{E}_M'z=b_1 \hat{E}_M''z,
\end{align}
where $\hat{a}''=(a_2,\ldots,a_{M+1})$, $\hat{b}''=(b_2,\ldots,b_{M+1})$,
\begin{align}
	\notag\hat{E}_M''&=\sum_{k=1}^{M+1}(-1)^k b_{1}^{M+1-k}[e_{k-1}(\hat{a}'')T^{-1}- q^{\lambda+1} e_{k-1}(\hat{b}'')]\prod_{n=0}^{M-k}(1-(a_{1}q^n/b_{1})T)\prod_{n=0}^{k-2}(1-q^{-n}T).
\end{align}
\begin{rem}
	When $M=1$, the degenerations $\hat{E}_1\to\hat{E}_1'$, $\hat{E}_1'\to\hat{E}_1''$ were discussed in \cite{HMST}.
	More precisely, {the equation $\hat{E}_1$ is equivalent to a variant of the $q$-hypergeometric equation of degree three,  $\hat{E}_1'$ is equivalent to it of degree two, and $\hat{E}_1''$ is equivalent to Heine's $q$-hypergeometric equation.}
\end{rem}
In addition, three term relations \eqref{sys3term1}--\eqref{sys3term5} are not changed for $2\leq k\neq l\leq M+1$.
Moreover, the $q$-shift operators $T_{a_{M+3}}$, $T_{b_{M+3}}$, $T_{a_{M+2}}$, $T_{b_{M+2}}$ become identities by taking the limits $a_{M+3}\to\infty$, $a_{M+2}\to0$ because
\begin{align}
	\lim_{a_{M+3}\to\infty}f(qa_{M+3})=\lim_{a_{M+3}\to\infty}f(a_{M+3}),\ \lim_{a_{M+2}\to0}f(qa_{M+2})=\lim_{a_{M+2}\to0}f(a_{M+2}),
\end{align}
if these limits converge.
Therefore the system $q$-$RP^M$ degenerates to the following system $q$-$RP^M_{\mathrm{degene}}$:
\begin{align}
	\label{hatem2dash}&\hat{E}_M''z=0,\\
	\label{multidegeneeq1}&[(a_{1}-a_k)T_{a_l}^{-1}-(a_l-a_k )T_{a_{1}}^{-1}+(a_l-a_{1})T_{a_k}^{-1}]z=0,\\
	\label{multidegeneeq4}&[(b_{1}-b_l)T_{b_k}-(b_k-b_l)T_{b_{1}}+(b_k-b_{1})T_{b_l}]z=0,\\
	&[(b_{1}-q^{-1}a_k)T_{b_l}-(b_l-q^{-1}a_k)T_{b_{1}}+(b_l-b_{1})T_{a_k}^{-1}]z=0,\\
	\label{multidegeneeq6}&[(qb_l-a_1)T_{a_k}^{-1}-(qb_l-a_k)T_{a_1}^{-1}+(a_1-a_k)T_{b_l}]z=0,\\
	\label{multidegeneeq7}&\left(q^{\lambda+1}\prod_{k=1}^{M+1}T_{a_k}T_{b_k}-1\right)z=0,
\end{align}
where $2\leq k\neq l\leq M+1$.
As a corollary, the integral 
\begin{align}\label{JORDANPOCHHAMMER}
	\int_0^{q/a_j}t^\lambda\prod_{i=1}^{M+1}\frac{(a_it)_\infty}{(b_it)_\infty}d_qt,
\end{align}
satisfies the above system if $|q^{\lambda+1}|<1$.
\begin{rem}
	The rank of $q$-$RP^M_{\mathrm{degene}}$ is $M+1$.
	This can be proved in the same way as the proof of Theorem \ref{thmmultisystem}.
	In \cite{Ma,Mi1989}, the Pfaff system related to the Jackson integral \eqref{JORDANPOCHHAMMER} was studied.
\end{rem}
\begin{rem}
	The relation \eqref{relationexponent} reduces $a_1\cdots a_{M+1}=q^{\lambda+\beta+2} b_1\cdots b_{M+1}$ by the degenerations.
	Since the parameter $\beta$ is independent from the system $q$-$RP^M_{\mathrm{degene}}$, there is no condition among $a_1,\ldots,a_{M+1},b_1,\ldots,b_{M+1},\lambda$.
\end{rem}
We will check that the system $q$-$RP^M_{\mathrm{degene}}$ is an extension of the $q$-Appell--Lauricella system.
\begin{prop}
	We put $a_{M+1}=q$, $b_{M+1}=C/A$, $q^\lambda=A/q$, $a_i=B_i x_i$, $b_i=x_i$ ($1\leq i\leq M$).
	Then the system $q$-$RP^M_{\mathrm{degene}}$ {implies} the $q$-Appell--Lauricella system
	\begin{align}
		&[x_i (1-T_{x_j})(1-B_i T_{x_i})-x_j(1-T_{x_i})(1-B_j T_{x_j})]z=0,\label{qALsyscap1}\\
		&[(1-T_{x_i})(1-Cq^{-1}T_{x_1}\cdots T_{x_M})-x_i (1-B_i T_{x_i})(1-AT_{x_1}\cdots T_{x_M})]z=0,\label{qALsyscap2}
	\end{align} 
	i.e. if the function $z$ satisfies $q$-$RP^M_{\mathrm{degene}}$, then $z$ is a solution for the $q$-Appell--Lauricella system.
\end{prop}
\begin{proof}
	For $1\leq i\leq M$,  the $q$-shift operator $T_{x_i}$ is rewritten as $T_{x_i}=T_{a_i}T_{b_i}$.
	We have
	\begin{align}
		\notag&[(1-T_{x_i})(1-Cq^{-1}T_{x_1}\cdots T_{x_M})-x_i (1-B_i T_{x_i})(1-AT_{x_1}\cdots T_{x_M})]z\\
		\notag&=[(1-T_{a_i}T_{b_i})(1-a_{M+1}^{-1}b_{M+1}q^{\lambda+1}T_{a_1}T_{b_1}\cdots T_{a_M}T_{b_M})\\
		&\quad-b_i(1-a_ib_i^{-1}T_{a_i}T_{b_i})(1-q^{\lambda+1}T_{a_1}T_{b_1}\cdots T_{a_M}T_{b_M})]z.
	\end{align}
	Using \eqref{multidegeneeq7}, we have
	\begin{align}
		\notag&(1-a_{M+1}^{-1}b_{M+1}q^{\lambda+1}T_{a_1}T_{b_1}\cdots T_{a_{M}}T_{b_M})z=(1-a_{M+1}^{-1}b_{M+1}T_{a_{M+1}}^{-1}T_{b_{M+1}}^{-1})z\\
		&=-T_{a_{M+1}}^{-1}T_{b_{M+1}}^{-1}(a_{M+1}^{-1}(b_{M+1}-a_{M+1}T_{a_{M+1}}T_{b_{M+1}}))z,\\
		&(1-q^{\lambda+1}T_{a_1}T_{b_1}\cdots T_{a_M}T_{b_M})z=-T_{a_{M+1}}^{-1}T_{b_{M+1}}^{-1}(b_{M+1}-a_{M+1}T_{a_{M+1}}T_{b_{M+1}}).
	\end{align}
	By the relations \eqref{multidegeneeq4} and \eqref{multidegeneeq6}, we have
	\begin{align}
		\notag&(qb_1-a_1)(1-T_{a_i}T_{b_i})z=T_{a_i}(qb_1-a_1)(T_{a_i}^{-1}-T_{b_i})z\\
		\notag&=T_{a_i}[(qb_1-a_i)T_{a_1}^{-1}-(a_1-a_i)T_{b_1}-(qb_i-a_1)T_{b_1}-q(b_i-b_1)T_{a_1}^{-1}]z\\
		&=T_{a_i}(qb_i-a_i)(T_{a_1}^{-1}-T_{b_1})z.
	\end{align}
	Similarly we find
	\begin{align}
		(qb_1-a_1)(b_i-a_iT_{a_i}T_{b_i})z=T_{a_i}(qb_i-a_i)(b_1T_{a_1}^{-1}-a_1q^{-1}T_{b_1})z.
	\end{align}
	Therefore we have
	\begin{align}
		\notag&(qb_1-a_1)^2[(1-T_{a_i}T_{b_i})a_{M+1}^{-1}(b_{M+1}-a_{M+1}T_{a_{M+1}}T_{b_{M+1}})-(b_i-a_iT_{a_i}T_{b_i})(1-T_{a_{M+1}}T_{b_{M+1}})]\\
		\notag&=T_{a_i}T_{a_{M+1}}(qb_i-a_i)(qb_{M+1}-a_{M+1})\\
		\notag&\phantom{=}\times[qa_{M+1}^{-1}(T_{a_1}^{-1}-T_{b_1})(b_1T_{a_1}^{-1}-a_1q^{-1}T_{b_1})-(b_1T_{a_1}^{-1}-a_1q^{-1}T_{b_1})(T_{a_1}^{-1}-T_{b_1})]\\
		\notag&=T_{a_i}T_{a_{M+1}}(qb_i-a_i)(qb_{M+1}-a_{M+1})(q^2 a_{M+1}^{-1}-1)(b_1T_{a_1}^{-1}-a_1q^{-1}T_{b_1})(T_{a_1}^{-1}-T_{b_1})\\
		&=(qb_i-qa_i)(qb_{M+1}-qa_{M+1})(q a_{M+1}^{-1}-1)T_{a_i}T_{a_{M+1}}(b_1T_{a_1}^{-1}-a_1q^{-1}T_{b_1})(T_{a_1}^{-1}-T_{b_1}).
	\end{align}
	Since $a_{M+1}=q$ and $qb_1-a_1\neq0$, the function $z$ satisfies
	\begin{align}
		[(1-T_{x_i})(1-Cq^{-1}T_{x_1}\cdots T_{x_M})-x_i (1-B_i T_{x_i})(1-AT_{x_1}\cdots T_{x_M})]z=0.
	\end{align}
	Also we have
	\begin{align}
		\notag&[x_i (1-T_{x_j})(1-B_i T_{x_i})-x_j(1-T_{x_i})(1-B_j T_{x_j})]z\\
		&=[(1-T_{a_j}T_{b_j})(b_i-a_i T_{a_i}T_{b_i})-(1-T_{a_i}T_{b_i})(b_j-a_j T_{a_j}T_{b_j})]z.
	\end{align}
	Similar to the above calculation, we have
	\begin{align}
		[x_i (1-T_{x_j})(1-B_i T_{x_i})-x_j(1-T_{x_i})(1-B_j T_{x_j})]z=0.
	\end{align}
	This completes the proof.
\end{proof}
\begin{rem}
	In the above proof, we did not use the equation \eqref{hatem2dash}.
	In general, the equation \eqref{hatem2dash} may be derived from \eqref{multidegeneeq1}--\eqref{multidegeneeq7}.
\end{rem}
\begin{cor}
	Suppose $\tau\in\{1,q/(B_1x_1),\ldots,q/(B_Mx_M),d\infty\}$, where $d\in\mathbb{C}\backslash\{0\}$ is an arbitrary point.
	{If $|A|<1$ (and $|qB_1B_2\cdots B_M/C|<1$ for $\tau=d\infty$ case),}
	then the integral
	\begin{align}
		\int_0^\tau t^{\alpha-1}\frac{(qt)_\infty}{(Ct/A)_\infty}\prod_{i=1}^{M}\frac{(B_ix_it)_\infty}{(x_it)_\infty}d_qt\quad (q^\alpha=A),
	\end{align} 
	is a solution of the $q$-Appell--Lauricella system.
\end{cor}



Finally we consider the degenerations for the series solution $W\left(\begin{array}{c}
	\{a_i\}_{1\leq i\leq M+3}\\\{b_i\}_{1\leq i\leq M+3}
\end{array}\right)$ of the system $q$-$RP^M$.
We recall the definition of $W$ \eqref{defW}:
{
\begin{align}
	\notag&W\left(\begin{array}{c}
		\{a_i\}_{1\leq i\leq M+3}\\
		\{b_i\}_{1\leq i\leq M+3}
	\end{array}\right)\\
	\notag&=
	\prod_{i=1}^M\frac{(qa_i/a_{M+2},qa_i/a_{M+3})_\infty}{(q^2a_ib_{M+3}/(a_{M+2}a_{M+3}))_\infty}\prod_{j=1}^{M+2}{(q^2b_jb_{M+3}/(a_{M+2}a_{M+3}))_\infty}\prod_{j=1}^{M+3}\frac{1}{(qb_j/a_{M+2},qb_j/a_{M+3})_\infty}\\
	\notag&\phantom{=}\times {(a_{M+1}/b_{M+3},a_{M+2}/a_{M+3},a_{M+3}/a_{M+2})_\infty}\frac{1}{a_{M+3}-a_{M+2}}\\
	&\phantom{=}\times W^{M,2}\left(\{a_i\}_{1\leq i\leq M};\frac{qb_{M+3}}{a_{M+2}a_{M+3}};\left\{\frac{1}{b_j}\right\}_{1\leq j\leq M+2};\frac{qb_{M+3}}{a_{M+2}},\frac{qb_{M+3}}{a_{M+3}};\frac{a_{M+1}}{b_{M+3}}\right),\\
	\notag&W^{M,2}\left(\{a_i\}_{1\leq i\leq M};\frac{qb_{M+3}}{a_{M+2}a_{M+3}};\left\{\frac{1}{b_j}\right\}_{1\leq j\leq M+2};\frac{qb_{M+3}}{a_{M+2}},\frac{qb_{M+3}}{a_{M+3}};\frac{a_{M+1}}{b_{M+3}}\right)\\
	\notag&=\sum_{l\in\mathbb{Z}_{\geq0}^M}\left(\frac{a_{M+1}}{b_{M+3}}\right)^{|l|}\frac{\Delta (\{a_iq^{l_i}\}_{1\leq i\leq M})}{\Delta(\{a_i\}_{1\leq i\leq M})}\frac{(qb_{M+3}/a_{M+2},qb_{M+3}/a_{M+3})_{|l|}}{\prod_{j=1}^{M+2}(q^2 b_jb_{M+3}/(a_{M+2}a_{M+3}))_{|l|}}\\
	&\phantom{=}\times\prod_{i=1}^M\left(\frac{1-(qb_{M+3}/(a_{M+2}a_{M+3}))a_iq^{|l|+l_i}}{1-(qb_{M+3}/(a_{M+2}a_{M+3}))a_i}\frac{(qa_ib_{M+3}/(a_{M+2}a_{M+3}))_{|l|}\prod_{j=1}^{M+2}(a_i/b_j)_{l_i}}{\prod_{j=1}^M(qa_i/a_j)_{l_i} \cdot(qa_i/a_{M+2},qa_i/a_{M+3})_{l_i}}\right).
\end{align}}
As discussed above, the system $q$-$RP^M$ is an extension of the $q$-Appell--Lauricella system.
The integral $\varphi_{i,j}$ \eqref{intmulti} is a $q$-analog of 
\begin{align}
	\int_{1/t_i}^{1/t_j} \prod_{k=1}^{M+3}(1-t_kt)^{\nu_i}dt\ \ (\nu_1+\cdots+\nu_{M+3}=-2).
\end{align}
{It is well-known in \cite{L} that the Appell--Lauricella hypergeometric series
\begin{align}
	&F_D\left(\begin{array}{c}
		\alpha;\{\beta_i\}_{1\leq i\leq M}\\
		\gamma
	\end{array};\{x_i\}_{1\leq i\leq M}\right)=\sum_{l\in(\mathbb{Z}_{\geq 0})^M}\frac{(\alpha;|l|)}{(\gamma;|l|)}\prod_{i=1}^M\left(\frac{(\beta_i;l_i)}{(1;l_i)}x_i^{l_i}\right)\quad (|x_i|<1),\\
	&(\alpha;k)=\frac{\Gamma(\alpha+k)}{\Gamma(\alpha)},
\end{align}
has the Euler type integral representation:
\begin{align}
	\label{difintFD}&F_D\left(\begin{array}{c}
		\alpha;\{\beta_i\}_{1\leq i\leq M}\\
		\gamma
	\end{array};\{x_i\}_{1\leq i\leq M}\right)=\frac{\Gamma(\gamma)}{\Gamma(\alpha)\Gamma(\gamma-\alpha)}\int_0^1 s^{\alpha-1}(1-s)^{\gamma-\alpha-1}\prod_{i=1}^M(1-x_is)^{-\beta_i}ds,
\end{align}
where $\mathrm{Re}\,(\alpha)>0$ and $\mathrm{Re}\,(\gamma-\alpha)>0$.
Changing the variables as
\begin{align}
	&x_i=\frac{t_i-t_{M+1}}{t_i-t_{M+3}}\frac{t_{M+2}-t_{M+3}}{t_{M+2}-t_{M+1}},\quad s=\frac{1-t_{M+3}t}{1-t_{M+1}t}\frac{t_{M+2}-t_{M+1}}{t_{M+2}-t_{M+3}},\\
	&\alpha=\nu_{M+3}+1,\quad \beta_i=-\nu_i,\quad \gamma=\nu_{M+2}+\nu_{M+3}+2,
\end{align}
we have
\begin{align}
	\notag&\int_{1/t_{M+3}}^{1/t_{M+2}}\prod_{k=1}^{M+3}(1-t_kt)^{\nu_k}dt\\
	\notag&=\frac{\Gamma(\nu_{M+2}+1)\Gamma(\nu_{M+3}+1)}{\Gamma(\nu_{M+2}+\nu_{M+3}+2)}\prod_{i=1}^M(t_i-t_{M+3})^{\nu_i}\\
	\notag&\phantom{=}\times(t_{M+1}-t_{M+3})^{\nu_{M+1}+\nu_{M+3}+2}(t-_{M+2}-t_{M+3})^{\nu_{M+2}+\nu_{M+3}+1}(t_{M+2}-t_{M+1})^{-\nu_{M+3}-1}\\
	&\phantom{=}\times F_D\left(\begin{array}{c}
		\nu_{M+3}+1;\{-\nu_i\}_{1\leq i\leq M}\\\nu_{M+2}+\nu_{M+3}+2
	\end{array};\left\{\frac{t_i-t_{M+1}}{t_i-t_{M+3}}\frac{t_{M+2}-t_{M+3}}{t_{M+2}-t_{M+1}}\right\}_{1\leq i\leq M}\right),
\end{align}
where $\nu_1+\nu_2+\cdots+\nu_{M+3}=-2$.}

{
Before considering degenerations of the series $W$, we show that the series $W^{M,2}$ is a $q$-analog of $\displaystyle F_D\left(\begin{array}{c}
	\nu_{M+3}+1;\{-\nu_i\}_{1\leq i\leq M}\\\nu_{M+2}+\nu_{M+3}+2
\end{array};\left\{\frac{t_i-t_{M+1}}{t_i-t_{M+3}}\frac{t_{M+2}-t_{M+3}}{t_{M+2}-t_{M+1}}\right\}_{1\leq i\leq M}\right)$.
We put $a_i=t_i$ and $b_i=q^{\nu_i}t_i$.
By using the formulas
\begin{align}
	\frac{(q^\alpha)_k}{(1-q)^k}\xrightarrow{q\to1}(\alpha;k),\ (q^\alpha X)_k\xrightarrow{q\to1}(1-X)^k,
\end{align}
we have
\begin{align}
	\notag&\left(\frac{a_{M+1}}{b_{M+3}}\right)^{|l|}\frac{\Delta (\{a_iq^{l_i}\}_{1\leq i\leq M})}{\Delta(\{a_i\}_{1\leq i\leq M})}\frac{(qb_{M+3}/a_{M+2},qb_{M+3}/a_{M+3})_{|l|}}{\prod_{j=1}^{M+2}(q^2 b_jb_{M+3}/(a_{M+2}a_{M+3}))_{|l|}}\\
	\notag&\times\prod_{i=1}^M\left(\frac{1-(qb_{M+3}/(a_{M+2}a_{M+3}))a_iq^{|l|+l_i}}{1-(qb_{M+3}/(a_{M+2}a_{M+3}))a_i}\frac{(qa_ib_{M+3}/(a_{M+2}a_{M+3}))_{|l|}\prod_{j=1}^{M+2}(a_i/b_j)_{l_i}}{\prod_{j=1}^M(qa_i/a_j)_{l_i} \cdot(qa_i/a_{M+2},qa_i/a_{M+3})_{l_i}}\right)\\
	&\xrightarrow{q\to1}\frac{(\nu_{M+3}+1;|l|)}{(\nu_{M+2}+\nu_{M+3}+2;|l|)}\prod_{i=1}^{M}\left(\frac{(-\nu_i;l_i)}{(1;l_i)}\left(\frac{t_i-t_{M+1}}{t_i-t_{M+3}}\frac{t_{M+2}-t_{M+3}}{t_{M+2}-t_{M+1}}\right)^{l_i}\right).
\end{align}
Therefore the series $\displaystyle W^{M,2}\left(\{a_i\}_{1\leq i\leq M};\frac{qb_{M+3}}{a_{M+2}a_{M+3}};\left\{\frac{1}{b_j}\right\}_{1\leq j\leq M+2};\frac{qb_{M+3}}{a_{M+2}},\frac{qb_{M+3}}{a_{M+3}};\frac{a_{M+1}}{b_{M+3}}\right)$ in \eqref{defW} can be regarded as a $q$-analog of $\displaystyle F_D\left(\begin{array}{c}
	\nu_{M+3}+1;\{-\nu_i\}_{1\leq i\leq M}\\\nu_{M+2}+\nu_{M+3}+2
\end{array};\left\{\frac{t_i-t_{M+1}}{t_i-t_{M+3}}\frac{t_{M+2}-t_{M+3}}{t_{M+2}-t_{M+1}}\right\}_{1\leq i\leq M}\right)$.}
%
%
The degeneration from $q$-$RP^M$ to the $q$-Appell--Lauricella system corresponds to the limit $\{t_{M+1},t_{M+2},t_{M+3}\}\to\{1,0,\infty\}$ roughly.
We will get series solutions, which are $q$-analogs of $F_D$ with linear fractions in the arguments, for the $q$-Appell--Lauricella system by taking a suitable limit of $W$.

Due to the integral representation of $W$ \eqref{inttoW}:
\begin{align}
	W\left(\begin{array}{c}
		\{a_i\}_{1\leq i\leq M+3}\\\{b_i\}_{1\leq i\leq M+3}
	\end{array}\right)=\frac{1}{q(1-q)(q)_\infty}\int_{q/a_{M+2}}^{q/a_{M+3}}\prod_{k=1}^{M+3}\frac{(a_kt)_\infty}{(b_kt)_\infty}d_qt,
\end{align}
and the equation \eqref{reductionintegral1}, we should consider the limit
\begin{align}
	\lim_{a_{M+3}\to\infty}\left(C_0\left(\frac{q}{a_{M+2}}\right)W\left(\begin{array}{c}
		\{a_{i}\}_{1\leq i\leq M+3}\\\{b_{i}\}_{1\leq i\leq M+3}
	\end{array}\right)\right)\bigg|_{b_{M+3}=q^\lambda a_{M+3}}.
\end{align}
By taking the termwise limit, we find
\begin{align}
	\notag&W^{M,2}\left(\{a_i\}_{1\leq i\leq M};\frac{qb_{M+3}}{a_{M+2}a_{M+3}};\left\{\frac{1}{b_j}\right\}_{1\leq j\leq M+2};\frac{qb_{M+3}}{a_{M+2}},\frac{qb_{M+3}}{a_{M+3}};\frac{a_{M+1}}{b_{M+3}}\right)\bigg|_{b_{M+3}=q^\lambda a_{M+3}}\\
	\notag&\xrightarrow{a_{M+3}\to\infty}\sum_{l\in\mathbb{Z}_{\geq0}^M}\left(-\frac{qa_{M+1}}{a_{M+2}}\right)^{|l|}q^{\binom{|l|}{2}}\frac{\Delta (\{a_iq^{l_i}\}_{1\leq i\leq M})}{\Delta(\{a_i\}_{1\leq i\leq M})}\frac{(q^{\lambda+1})_{|l|}}{\prod_{j=1}^{M+2}(q^{\lambda+2} b_j/a_{M+2})_{|l|}}\\
	\label{just6}&\phantom{\xrightarrow{q\to1}\sum_{l\in\mathbb{Z}_{\geq0}^M}}\times\prod_{i=1}^M\left(\frac{1-(q^{\lambda+1} /a_{M+2})a_iq^{|l|+l_i}}{1-(q^{\lambda+1} /a_{M+2})a_i}\frac{(q^{\lambda+1} a_i/a_{M+2})_{|l|}\prod_{j=1}^{M+2}(a_i/b_j)_{l_i}}{\prod_{j=1}^M(qa_i/a_j)_{l_i} \cdot(qa_i/a_{M+2})_{l_i}}\right),\\
	\notag&\left(s(b_{M+2},b_{M+3}).W^{M,2}\left(\{a_i\}_{1\leq i\leq M};\frac{qb_{M+3}}{a_{M+2}a_{M+3}};\left\{\frac{1}{b_j}\right\}_{1\leq j\leq M+2};\frac{qb_{M+3}}{a_{M+2}},\frac{qb_{M+3}}{a_{M+3}};\frac{a_{M+1}}{b_{M+3}}\right)\right)\bigg|_{b_{M+3}=q^\lambda a_{M+3}}\\
	\label{just7}&\xrightarrow{a_{M+3}\to\infty}\sum_{l\in\mathbb{Z}_{\geq0}^M}\left(\frac{a_{M+1}}{b_{M+2}}\right)^{|l|}\frac{\Delta (\{a_iq^{l_i}\}_{1\leq i\leq M})}{\Delta(\{a_i\}_{1\leq i\leq M})}\frac{(qb_{M+2}/a_{M+2})_{|l|}}{(q^{\lambda+2} b_{M+2}/a_{M+2})_{|l|}}\prod_{i=1}^M\left(\frac{\prod_{j=1}^{M+1}(a_i/b_j)_{l_i}}{\prod_{j=1}^M(qa_i/a_j)_{l_i} \cdot(qa_i/a_{M+2})_{l_i}}\right),
\end{align}
{where $s(x,y):x\leftrightarrow y$.
These termwise limits are justified by Tannery's theorem (see Appendix \ref{appA}).
Hence, we have
\begin{align}
	\notag&\lim_{a_{M+3}\to\infty}\left(C_0\left(\frac{q}{a_{M+2}}\right)W\left(\begin{array}{c}
		\{a_{i}\}_{1\leq i\leq M+3}\\\{b_{i}\}_{1\leq i\leq M+3}
	\end{array}\right)\right)\bigg|_{b_{M+3}=q^\lambda a_{M+3}}\\
	\notag&=\frac{-1}{a_{M+2}(q^{\lambda+1})_\infty}\left(\frac{q}{a_{M+2}}\right)^\lambda\prod_{i=1}^M\frac{(qa_i/a_{M+2})_\infty}{(q^{\lambda+2} a_i/a_{M+2})_\infty}\prod_{j=1}^{M+2}\frac{(q^{\lambda+2} b_j/a_{M+2})_\infty}{(qb_j/a_{M+2})_\infty}\\
	\notag&\phantom{=}\times\sum_{l\in\mathbb{Z}_{\geq0}^M}\left(-\frac{qa_{M+1}}{a_{M+2}}\right)^{|l|}q^{\binom{|l|}{2}}\frac{\Delta (\{a_iq^{l_i}\}_{1\leq i\leq M})}{\Delta(\{a_i\}_{1\leq i\leq M})}\frac{(qq^\lambda)_{|l|}}{\prod_{j=1}^{M+2}(q^{\lambda+2} b_j/a_{M+2})_{|l|}}\\
	&\phantom{\xrightarrow{q\to1}\sum_{l\in\mathbb{Z}_{\geq0}^M}}\times\prod_{i=1}^M\left(\frac{1-(q^{\lambda+1} /a_{M+2})a_iq^{|l|+l_i}}{1-(q^{\lambda+1} /a_{M+2})a_i}\frac{(q^{\lambda+1} a_i/a_{M+2})_{|l|}\prod_{j=1}^{M+2}(a_i/b_j)_{l_i}}{\prod_{j=1}^M(qa_i/a_j)_{l_i} \cdot(qa_i/a_{M+2})_{l_i}}\right),\label{newdegene1}
\end{align}
and if $|a_{M+1}/b_{M+2}|<1$, we have
\begin{align}
	\notag&\lim_{a_{M+3}\to\infty}\left(C_0\left(\frac{q}{a_{M+2}}\right)s(b_{M+2},b_{M+3}).W\left(\begin{array}{c}
		\{a_{i}\}_{1\leq i\leq M+3}\\\{b_{i}\}_{1\leq i\leq M+3}
	\end{array}\right)\right)\bigg|_{b_{M+3}=q^\lambda a_{M+3}}\\
	\notag&=\frac{-(q^{\lambda+2}b_{M+2}/a_{M+2},a_{M+1}/b_{M+2})_\infty}{a_{M+2}(q^{\lambda+1})_\infty}\left(\frac{q}{a_{M+2}}\right)^\lambda\prod_{i=1}^M{(qa_i/a_{M+2})_\infty}\prod_{j=1}^{M+2}\frac{1}{(qb_j/a_{M+2})_\infty}\\
	&\phantom{=}\times\sum_{l\in\mathbb{Z}_{\geq0}^M}\left(\frac{a_{M+1}}{b_{M+2}}\right)^{|l|}\frac{\Delta (\{a_iq^{l_i}\}_{1\leq i\leq M})}{\Delta(\{a_i\}_{1\leq i\leq M})}\frac{(qb_{M+2}/a_{M+2})_{|l|}}{(q^{\lambda+2} b_{M+2}/a_{M+2})_{|l|}}\prod_{i=1}^M\left(\frac{\prod_{j=1}^{M+1}(a_i/b_j)_{l_i}}{\prod_{j=1}^M(qa_i/a_j)_{l_i} \cdot(qa_i/a_{M+2})_{l_i}}\right).\label{newdegene2}
\end{align}
}
In the above calculations, we use $(1/x)_k\xrightarrow{x\to\infty}1$ $(k\in\mathbb{Z}_{\geq0}\cup\{\infty\})$, $(x)_k x^{-k}\xrightarrow{x\to\infty}(-1)^kq^{\binom{k}{2}}$ $(k\in\mathbb{Z}_{\geq0})$.
\begin{rem}
	The series \eqref{newdegene1} converges for any $a_i$, $b_i$ and $\lambda$.
	The key factor is $q^{\binom{l}{2}}$.
	See Appendix \ref{appA} for more details.
\end{rem}
Next we consider the limit $a_{M+2}\to0$ with $b_{M+2}=q^\beta a_{M+2}$.
{We calculate three limits,
\begin{align}
	&\lim_{a_{M+2}\to0} (s(a_{M+1},a_{M+2}).\mbox{right-hand side of \eqref{newdegene1}})|_{b_{M+2}=q^\beta a_{M+2}},\\
	&\lim_{a_{M+2}\to0} (s(a_{M+1},a_{M+2}).\mbox{right-hand side of \eqref{newdegene2}})|_{b_{M+2}=q^\beta a_{M+2}},\\
	&\lim_{a_{M+2}\to0} (s(a_{M+1},a_{M+2}).s(b_{M+1},b_{M+2}).\mbox{right-hand side of \eqref{newdegene2}})|_{b_{M+2}=q^\beta a_{M+2}}.
\end{align}
Note that $s(a_{M+1},a_{M+2}).y$, $s(a_{M+1},a_{M+2}).s(b_{M+1},b_{M+2}).y$ are solutions of $q$-$RP^M$ if $y$ is a solution, because the system $q$-$RP^M$ is symmetric for $\{a_i\}$, and for $\{b_i\}$.
Using the formulas $(x)_k\xrightarrow{x\to0}1$ $(k\in\mathbb{Z}_{\geq0}\cup\{\infty\})$, $(1/x)_k x^k\xrightarrow{x\to0}(-1)^k q^{\binom{k}{2}}$ $(k\in\mathbb{Z}_{\geq0})$, we have
\begin{align}
	\notag&\lim_{a_{M+2}\to0} (s(a_{M+1},a_{M+2}).\mbox{right-hand side of \eqref{newdegene1}})|_{b_{M+2}=q^\beta a_{M+2}}\\
	\notag&=\frac{-1}{a_{M+1}(q^{\lambda+1})_\infty}\left(\frac{q}{a_{M+1}}\right)^\lambda\prod_{i=1}^M\frac{(qa_i/a_{M+1})_\infty}{(q^{\lambda+2} a_i/a_{M+1})_\infty}\prod_{j=1}^{M+1}\frac{(q^{\lambda+2} b_j/a_{M+1})_\infty}{(qb_j/a_{M+1})_\infty}\\
	\notag&\phantom{=}\times\sum_{l\in\mathbb{Z}_{\geq0}^M}\left(-\frac{q}{a_{M+1}}\right)^{|l|}q^{\binom{|l|}{2}}\frac{\Delta (\{a_iq^{l_i}\}_{1\leq i\leq M})}{\Delta(\{a_i\}_{1\leq i\leq M})}\frac{(q^{\lambda+1})_{|l|}}{\prod_{j=1}^{M+1}(q^{\lambda+2} b_j/a_{M+1})_{|l|}}\\
	\label{just8}&\phantom{\xrightarrow{q\to1}\sum_{l\in\mathbb{Z}_{\geq0}^M}}\times\prod_{i=1}^M\left(\frac{1-(q^{\lambda+1} /a_{M+1})a_iq^{|l|+l_i}}{1-(q^{\lambda+1} /a_{M+1})a_i}\frac{(q^{\lambda+1} a_i/a_{M+1})_{|l|}\prod_{j=1}^{M+1}(a_i/b_j)_{l_i}\cdot (-a_i/q^\beta)^{l_i}q^{\binom{l_i}{2}}}{\prod_{j=1}^{M+1}(qa_i/a_j)_{l_i} }\right),\\
	\notag&\lim_{a_{M+2}\to0} (s(a_{M+1},a_{M+2}).\mbox{right-hand side of \eqref{newdegene2}})|_{b_{M+2}=q^\beta a_{M+2}}\\
	\notag&=\frac{-(q^{-\beta})_\infty}{a_{M+1}(q^{\lambda+1})_\infty}\left(\frac{q}{a_{M+1}}\right)^\lambda\prod_{i=1}^M{(qa_i/a_{M+1})_\infty}\prod_{j=1}^{M+1}\frac{1}{(qb_j/a_{M+1})_\infty}\\
	\label{just9}&\phantom{=}\times\sum_{l\in\mathbb{Z}_{\geq0}^M}\left(q^{-\beta}\right)^{|l|}\frac{\Delta (\{a_iq^{l_i}\}_{1\leq i\leq M})}{\Delta(\{a_i\}_{1\leq i\leq M})}\prod_{i=1}^M\left(\frac{\prod_{j=1}^{M+1}(a_i/b_j)_{l_i}}{\prod_{j=1}^{M+1}(qa_i/a_j)_{l_i} }\right)\qquad(|q^{-\beta}|<1),\\
	\notag&\lim_{a_{M+2}\to0} (s(a_{M+1},a_{M+2}).s(b_{M+1},b_{M+2}).\mbox{right-hand side of \eqref{newdegene2}})|_{b_{M+2}=q^\beta a_{M+2}}\\
	\notag&=\frac{-(q^{\lambda+2}b_{M+1}/a_{M+1})_\infty}{a_{M+1}(q^{\lambda+1})_\infty}\left(\frac{q}{a_{M+1}}\right)^\lambda\prod_{i=1}^M{(qa_i/a_{M+1})_\infty}\prod_{j=1}^{M+1}\frac{1}{(qb_j/a_{M+1})_\infty}\\
	\label{just10}&\phantom{=}\times\sum_{l\in\mathbb{Z}_{\geq0}^M}\left(\frac{1}{b_{M+1}}\right)^{|l|}\frac{\Delta (\{a_iq^{l_i}\}_{1\leq i\leq M})}{\Delta(\{a_i\}_{1\leq i\leq M})}\frac{(qb_{M+1}/a_{M+1})_{|l|}}{(q^{\lambda+2} b_{M+1}/a_{M+1})_{|l|}}\prod_{i=1}^M\left(\frac{\prod_{j=1}^{M}(a_i/b_j)_{l_i}\cdot (-a_i/q^\beta)^{l_i}q^{\binom{l_i}{2}}}{\prod_{j=1}^{M+1}(qa_i/a_j)_{l_i}}\right).
\end{align}
These limits can be calculated by Tannery's theorem (see Appendix \ref{appA}).}
In conclusion, we get three solutions for the system $q$-$RP^M_{\mathrm{degene}}$.
\begin{prop}
	The following functions satisfy the system $q$-$RP^M_{\mathrm{degene}}$:
	\begin{align}
		\notag&\left(\frac{1}{a_{M+1}}\right)^{\lambda+1}\prod_{i=1}^M\frac{(qa_i/a_{M+1})_\infty}{(q^{\lambda+2} a_i/a_{M+1})_\infty}\prod_{j=1}^{M+1}\frac{(q^{\lambda+2} b_j/a_{M+1})_\infty}{(qb_j/a_{M+1})_\infty}\\
		\notag&\phantom{=}\times\sum_{l\in\mathbb{Z}_{\geq0}^M}\left(\frac{q}{a_{M+1}}\right)^{|l|}q^{\binom{|l|}{2}}\frac{\Delta (\{a_iq^{l_i}\}_{1\leq i\leq M})}{\Delta(\{a_i\}_{1\leq i\leq M})}\frac{(q^{\lambda+1})_{|l|}}{\prod_{j=1}^{M+1}(q^{\lambda+2} b_j/a_{M+1})_{|l|}}\\
		\label{newdegene11}&\phantom{\xrightarrow{q\to1}\sum_{l\in\mathbb{Z}_{\geq0}^M}}\times\prod_{i=1}^M\left(\frac{1-(q^{\lambda+1} /a_{M+1})a_iq^{|l|+l_i}}{1-(q^{\lambda+1} /a_{M+1})a_i}\frac{(q^{\lambda+1} a_i/a_{M+1})_{|l|}\prod_{j=1}^{M+1}(a_i/b_j)_{l_i}\cdot (a_i/q^\beta)^{l_i}q^{\binom{l_i}{2}}}{\prod_{j=1}^{M+1}(qa_i/a_j)_{l_i} }\right),\\
		\notag&\left(\frac{1}{a_{M+1}}\right)^{\lambda+1}\prod_{i=1}^M{(qa_i/a_{M+1})_\infty}\prod_{j=1}^{M+1}\frac{1}{(qb_j/a_{M+1})_\infty}\\
		\label{newdegene21}&\phantom{=}\times\sum_{l\in\mathbb{Z}_{\geq0}^M}\left(q^{-\beta}\right)^{|l|}\frac{\Delta (\{a_iq^{l_i}\}_{1\leq i\leq M})}{\Delta(\{a_i\}_{1\leq i\leq M})}\prod_{i=1}^M\left(\frac{\prod_{j=1}^{M+1}(a_i/b_j)_{l_i}}{\prod_{j=1}^{M+1}(qa_i/a_j)_{l_i} }\right)\qquad(|q^{-\beta}|<1),\\
		\notag&\frac{(q^{\lambda+2}b_{M+1}/a_{M+1})_\infty}{(q^{\lambda+1})_\infty}\left(\frac{1}{a_{M+1}}\right)^{\lambda+1}\prod_{i=1}^M{(qa_i/a_{M+1})_\infty}\prod_{j=1}^{M+1}\frac{1}{(qb_j/a_{M+1})_\infty}\\
		\label{newdegene22}&\phantom{=}\times\sum_{l\in\mathbb{Z}_{\geq0}^M}\left(\frac{1}{b_{M+1}}\right)^{|l|}\frac{\Delta (\{a_iq^{l_i}\}_{1\leq i\leq M})}{\Delta(\{a_i\}_{1\leq i\leq M})}\frac{(qb_{M+1}/a_{M+1})_{|l|}}{(q^{\lambda+2} b_{M+1}/a_{M+1})_{|l|}}\prod_{i=1}^M\left(\frac{\prod_{j=1}^{M}(a_i/b_j)_{l_i}\cdot (-a_i/q^\beta)^{l_i}q^{\binom{l_i}{2}}}{\prod_{j=1}^{M+1}(qa_i/a_j)_{l_i}}\right).
	\end{align}
	Here $q^\beta=a_1\cdots a_{M+1}/(q^{\lambda+2}b_1\cdots b_{M+1})$.
\end{prop}
\begin{rem}
	The series \eqref{newdegene11} and \eqref{newdegene22} converge for arbitrary values thanks to the factor $q^{\binom{l}{2}}$ and $q^{\binom{l_i}{2}}$.
	See Appendix \ref{appA} for more details.
\end{rem}
\begin{cor}\label{newcorALsol}
	The functions
	\begin{align}
		\notag&\prod_{i=1}^M \frac{(Ax_i,B_ix_i)_\infty}{(AB_ix_i,x_i)_\infty}
		\\
		\notag&\times \sum_{l\in\mathbb{Z}_{\geq0}^M}q^{\binom{|l|}{2}}\frac{\Delta(\{B_ix_iq^{l_i}\}_{1\leq i\leq M})}{\Delta(\{B_ix_i\}_{1\leq i\leq M})}\frac{(A)_{|l|}}{(C)_{|l|}}\\
		&\phantom{\xrightarrow{q\to1}\sum_{l\in\mathbb{Z}_{\geq0}^M}}\times\prod_{i=1}^M \left(\frac{1-A B_ix_iq^{|l|+l_i}/q}{1-A B_ix_i/q}\frac{(AB_ix_i/q)_{|l|}\prod_{j=1}^M (B_ix_i/x_j)_{l_i}\cdot (AB_ix_i/C)_{l_i}(B_iCx_i/B)^{l_i}q^{\binom{l_i}{2}}}{(Ax_i)_{|l|}\prod_{j=1}^M(qB_ix_i/(B_jx_j))_{l_i}\cdot (B_ix_i)_{l_i}}\right),\label{newqALsol1}\\
		&\prod_{i=1}^M\frac{(B_ix_i)_\infty}{(x_i)_\infty}\sum_{l\in\mathbb{Z}_{\geq0}^M}\left(\frac{C}{B}\right)^{|l|}\frac{\Delta(\{B_ix_iq^{l_i}\}_{1\leq i\leq M})}{\Delta(\{B_ix_i\}_{1\leq i\leq M})}\prod_{i=1}^M \frac{\prod_{j=1}^M(B_ix_i/x_j)_{l_i}\cdot (AB_ix_i/C)_{l_i}}{\prod_{j=1}^M(qB_ix_i/(B_jx_j))_{l_i}\cdot (B_ix_i)_{l_i}}\quad \left(\left|\frac{C}{B}\right|<1\right),\label{newqALsol2}\\
		&\prod_{i=1}^M\frac{(B_ix_i)_\infty}{(x_i)_\infty}\sum_{l\in\mathbb{Z}_{\geq0}^M}\left(-\frac{A}{B}\right)^{|l|}\frac{\Delta(\{B_ix_iq^{l_i}\}_{1\leq i\leq M})}{\Delta(\{B_ix_i\}_{1\leq i\leq M})}\frac{(C/A)_{|l|}}{(C)_{|l|}}\prod_{i=1}^M\frac{\prod_{j=1}^M(B_ix_i/x_j)_{l_i}\cdot (B_ix_i)^{l_i}q^{\binom{l_i}{2}}}{\prod_{j=1}^M (qB_ix_i/(B_jx_j))_{l_i}\cdot (B_ix_i)_{l_i}},\label{newqALsol3}
	\end{align}
	where $B=B_1B_2\cdots B_M$,
	are solutions of the $q$-Appell--Lauricella system \eqref{qALsyscap1}, \eqref{qALsyscap2}.
\end{cor}
\begin{proof}
	We get desired results by putting $a_i=B_i x_i$, $b_i=x_i$ ($1\leq i\leq M$), $a_{M+1}=q$, $b_{M+1}=C/A$, $q^\lambda=A/q$.
	Note that $q^\beta=q^{-\lambda-2} a_1\cdots a_{M+1}/(b_1\cdots b_{M+1})=B_1\cdots B_M/C$
\end{proof}
In \cite{N}, a characterization of solutions for the $q$-Appell--Lauricella system was obtained.
\begin{prop}[\cite{N}]
	{Suppose that} a solution $y$ for the $q$-Appell--Lauricella system \eqref{qALsyscap1}, \eqref{qALsyscap2} has the following asymptotic behavior near $\{0\ll|x_1|\ll\cdots \ll|x_M|\ll1\}$:
	\begin{align}
		y=1+O(|x_1/x_2|,|x_2/x_3|,\ldots,|x_{M-1}/x_M|,|x_M|).
	\end{align}
	Then we have 
	\begin{align}
		y=\varphi_D\left(\begin{array}{c}
			A;B_1,\ldots,B_M\\
			C
		\end{array};x_1,\ldots,x_M\right).
	\end{align}
\end{prop}
Using this characterization, we find relations between the $q$-Appell--Lauricella series $\varphi_D$ and the above solutions \eqref{newqALsol1}, \eqref{newqALsol2}, \eqref{newqALsol3}.
\begin{prop}\label{newpropALtrans}
	Suppose $|x_i|<1$ $(1\leq i\leq M)$.
	We have
	\begin{align}
		\notag&\prod_{i=1}^M \frac{(Ax_i,B_ix_i)_\infty}{(AB_ix_i,x_i)_\infty}
		\\
		\notag&\times \sum_{l\in\mathbb{Z}_{\geq0}^M}q^{\binom{|l|}{2}}\frac{\Delta(\{B_ix_iq^{l_i}\}_{1\leq i\leq M})}{\Delta(\{B_ix_i\}_{1\leq i\leq M})}\frac{(A)_{|l|}}{ (C)_{|l|}}\\
		\notag&\phantom{\xrightarrow{q\to1}\sum_{l\in\mathbb{Z}_{\geq0}^M}}\times\prod_{i=1}^M \left(\frac{1-A B_ix_iq^{|l|+l_i}/q}{1-A B_ix_i/q}\frac{(AB_ix_i/q)_{|l|}\prod_{j=1}^M (B_ix_i/x_j)_{l_i}\cdot (AB_ix_i/C)_{l_i}(B_iCx_i/B)^{l_i}q^{\binom{l_i}{2}}}{(Ax_i)_{|l|}\prod_{j=1}^M(qB_ix_i/(B_jx_j))_{l_i}\cdot (B_ix_i)_{l_i}}\right)\\
		\label{newtransAL1}&=\varphi_D\left(\begin{array}{c}
			A;\{B_i\}_{1\leq i\leq M}\\C
		\end{array};\{x_i\}\right),\\
		\notag&\frac{(C/B)_\infty}{(C)_\infty}\prod_{i=1}^M\frac{(B_ix_i)_\infty}{(x_i)_\infty}\sum_{l\in\mathbb{Z}_{\geq0}^M}\left(\frac{C}{B}\right)^{|l|}\frac{\Delta(\{B_ix_iq^{l_i}\}_{1\leq i\leq M})}{\Delta(\{B_ix_i\}_{1\leq i\leq M})}\prod_{i=1}^M \frac{\prod_{j=1}^M(B_ix_i/x_j)_{l_i}\cdot (AB_ix_i/C)_{l_i}}{\prod_{j=1}^M(qB_ix_i/(B_jx_j))_{l_i}\cdot (B_ix_i)_{l_i}}\\
		\label{newtransAL2}&=\varphi_D\left(\begin{array}{c}
			A;\{B_i\}_{1\leq i\leq M}\\C
		\end{array};\{x_i\}\right)\quad \left(\left|\frac{C}{B}\right|<1\right),\\
		\notag&\prod_{i=1}^M\frac{(B_ix_i)_\infty}{(x_i)_\infty}\sum_{l\in\mathbb{Z}_{\geq0}^M}\left(-\frac{A}{B}\right)^{|l|}\frac{\Delta(\{B_ix_iq^{l_i}\}_{1\leq i\leq M})}{\Delta(\{B_ix_i\}_{1\leq i\leq M})}\frac{(C/A)_{|l|}}{(C)_{|l|}}\prod_{i=1}^M\frac{\prod_{j=1}^M(B_ix_i/x_j)_{l_i}\cdot (B_ix_i)^{l_i}q^{\binom{l_i}{2}}}{\prod_{j=1}^M (qB_ix_i/(B_jx_j))_{l_i}\cdot (B_ix_i)_{l_i}}\\
		\label{newtransAL3}&=\varphi_D\left(\begin{array}{c}
			A;\{B_i\}_{1\leq i\leq M}\\C
		\end{array};\{x_i\}\right),
	\end{align}
	where $B=B_1B_2\cdots B_M$.
\end{prop}
\begin{proof}
	We put $X_1=x_1/x_2$, $X_2=x_2/x_3$, $\ldots$, $X_{M-1}=x_{M-1}/x_M$, $X_M=x_M$, then we have
	\begin{align}
		&\frac{1-AB_ix_iq^{|l|+l_i}/q}{1-AB_ix_i/q}=1+O(|X_1|+\cdots+|X_M|).
	\end{align}
	Due to
	\begin{align}
		&{\Delta(\{B_ix_i\}_{1\leq i\leq M})}=\prod_{1\leq i<j\leq M}(B_i X_i\cdots X_M-B_jX_j\cdots X_M)=\prod_{1\leq i<j\leq M}X_j\cdots X_M(B_iX_i\cdots X_{j-1}-B_j),
	\end{align}
	we have
	\begin{align}
		\frac{\Delta(\{B_ix_iq^{l_i}\}_{1\leq i\leq M})}{\Delta(\{B_ix_i\}_{1\leq i\leq M})}=\prod_{1\leq i<j\leq M}\frac{q^{l_j}B_j-q^{l_i}B_iX_i\cdots X_{j-1}}{B_j-B_iX_i\cdots X_{j-1}}=\prod_{1\leq i<j\leq M}q^{l_j}+O(|X_1|+\cdot+|X_M|).
	\end{align}
	For $i<j$, we have
	\begin{align}
		\frac{(B_ix_i/x_j)_{l_i}}{(qB_ix_i/(B_jx_j))_{l_i}}=\frac{(B_iX_i\cdots X_{j-1})_{l_i}}{(qB_iX_i\cdots X_{j-1}/B_j)_{l_i}}=1+O(|X_1|+\cdots+|X_M|).
	\end{align}
	For $i>j$, we have
	\begin{align}
		\frac{(B_ix_i/x_j)_{l_i}}{(qB_ix_i/(B_jx_j))_{l_i}}
		=\left(\frac{B_j}{q}\right)^{l_i}\frac{(q^{1-l_i}x_j/(B_ix_i))_{l_i}}{(q^{-l_i}B_jx_j/(B_ix_i))_{l_i}}
		=\left(\frac{B_j}{q}\right)^{l_i}+O(|X_1|+\cdots+|X_M|).
	\end{align}
	In addition, the function $(Dx_i)_{l}$, where $D$ is a constant for $x_1,\ldots,x_M$, is expanded as
	\begin{align}
		(Dx_i)_{l}=1+O(|X_1|+\cdots+|X_M|)\quad (l\in\mathbb{Z}\cup\{\infty\}).
	\end{align}
	Therefore we obtain
	\begin{align}
		&(\mbox{solution \eqref{newqALsol1}})=1+O(|X_1|+\cdots+|X_M|),\\
		&(\mbox{solution \eqref{newqALsol2}})=\sum_{l\in\mathbb{Z}_{\geq0}^M}\left(\frac{C}{B}\right)^{|l|}\prod_{i=1}^M\frac{(B_i)_{l_i}}{(q)_{l_i}}\prod_{1\leq i<j\leq M}B_i^{l_j}+O(|X_1|+\cdots+|X_M|)\notag\\
		&=\prod_{i=1}^M\sum_{l_i=0}^\infty\frac{(B_i)_{l_i}}{(q)_{l_i}}\left(\frac{C}{B_i\cdots B_M}\right)^{l_i}+O(|X_1|+\cdots+|X_M|)=\frac{(C)_\infty}{(B)_\infty}+O(|X_1|+\cdots+|X_M|),\\
		&(\mbox{solution \eqref{newqALsol3}})=1+O(|X_1|+\cdots+|X_M|).
	\end{align}
	This completes the proof.
\end{proof}
\begin{rem}
	The formula \eqref{newtransAL2} is equivalent to the following transformation \cite{Kaji}:
	\begin{align}
		\notag&\frac{(u)_\infty}{(a_1\cdots a_Mu)_\infty}\prod_{1\leq i\leq M}\frac{(cx_i/x_M)_\infty}{((c/a_i)x_i/x_M)_\infty}\sum_{l\in(\mathbb{Z}_{\geq0})^M}u^{|l|}\frac{\Delta(xq^l)}{\Delta(x)}\prod_{1\leq i,j\leq M}\frac{(a_jx_i/x_j)_{l_i}}{(qx_i/x_j)_{l_i}}\prod_{1\leq i\leq M}\frac{(bx_i/x_M)_{l_i}}{(cx_i/x_M)_{l_i}}\\
		&=\varphi_D\left(\begin{array}{c}
			a_1\cdots a_Mbu/c;a_1,\ldots,a_M\\
			a_1\cdots a_M u
		\end{array};\frac{c}{a_1}\frac{x_1}{x_M},\ldots,\frac{c}{a_{M-1}}\frac{x_{M-1}}{x_M},\frac{c}{a_M}\right).\label{remkajitransgivenqAL}
	\end{align}
	This formula was obtained in \cite{Kaji} by using some transformation formula for multiple $q$-hypergeometric series and Andrews' formula \cite{And}:
	\begin{align}
		\varphi_D\left(\begin{array}{c}
			A;\{B_i\}_{1\leq i\leq M}\\C
		\end{array};\{x_i\}\right)=\frac{(A)_\infty}{(C)_\infty}\prod_{i=1}^M\frac{(B_ix_i)_\infty}{(x_i)_\infty}{}_{M+1}\varphi_M\left(\begin{array}{c}
			C/A,\{x_i\}_{1\leq i\leq M}\\\{B_ix_i\}_{1\leq i\leq M}
		\end{array};A\right),\label{andrews}
	\end{align}
	where $|x_i|<1$ and $|A|<1$.
	On the other hand, we get \eqref{newtransAL2} from the viewpoint of $q$-difference equations.
\end{rem}
\begin{rem}
	When $M=1$, the formulas \eqref{newtransAL2} and \eqref{newtransAL3} reduce to 2nd Heine's and Jackson's transformation formula \cite{GR}:
	\begin{align}
		{}_2\varphi_1\left(\begin{array}{c}
			A,B\\C
		\end{array};x\right)&=\frac{(C/B,Bx)_\infty}{(C,x)_\infty}{}_2\varphi_1\left(\begin{array}{c}
			ABx/C,B\\Bx
		\end{array};\frac{C}{B}\right),\\
		&=\frac{(Bx)_\infty}{(x)_\infty}{}_2\varphi_2\left(\begin{array}{c}
			B,C/A\\C,Bx
		\end{array};Ax\right).
	\end{align}
\end{rem}
\begin{rem}
	The formula \eqref{newtransAL3} is a $q$-analog of
	\begin{align}
		F_D\left(\begin{array}{c}
			\alpha;\beta_1\ldots,\beta_M\\
			\gamma
		\end{array};x_1,\ldots,x_M\right)=\prod_{i=1}^M(1-x_i)^{-\beta_i}F_D\left(\begin{array}{c}
			\gamma-\alpha;\beta_1,\ldots,\beta_M\\
			\gamma
		\end{array};\frac{x_1}{x_1-1},\ldots,\frac{x_M}{x_M-1}\right).
	\end{align}
	This formula can be derived by changing $s\to1-s$ in the integral representation \eqref{difintFD} of $F_D$.
\end{rem}
{Using Andrews' formula \eqref{andrews} and putting
	\begin{align}
		A=x,\ B_i=b_i/a_i,\ C=a_{M+1}x,\ x_i=a_i,
	\end{align}
	the series \eqref{newqALsol1}, \eqref{newqALsol2}, \eqref{newqALsol3} can be transformed to ${}_{M+1}\varphi_M\left(\begin{array}{c}
		\{a_i\}_{1\leq i\leq M+1}\\\{b_i\}_{1\leq i\leq M}
	\end{array};x\right)$ as follows.}
\begin{cor}
	Suppose $|x|<1$.
	We have
	\begin{align}
		&\frac{1}{(x)_\infty}\prod_{i=1}^{M+1} (a_ix)_\infty\prod_{i=1}^M \frac{1}{(b_ix)_\infty}
		\\
		\notag&\times \sum_{l\in\mathbb{Z}_{\geq0}^M}q^{\binom{|l|}{2}}\frac{\Delta(\{b_iq^{l_i}\}_{1\leq i\leq M})}{\Delta(\{b_i\}_{1\leq i\leq M})}\frac{(x)_{|l|}}{ (a_{M+1}x)_{|l|}}\\
		\notag&\phantom{\xrightarrow{q\to1}\sum_{l\in\mathbb{Z}_{\geq0}^M}}\times\prod_{i=1}^M \left(\frac{1-b_ixq^{|l|+l_i}/q}{1-b_ix/q}\frac{(b_ix/q)_{|l|}\prod_{j=1}^{M+1} (b_i/a_j)_{l_i}\cdot(Ab_ix/B)^{l_i}q^{\binom{l_i}{2}}}{(a_ix)_{|l|}\prod_{j=1}^M(qb_i/b_j)_{l_i}\cdot (b_i)_{l_i}}\right)\\
		&={}_{M+1}\varphi_M\left(\begin{array}{c}
			\{a_i\}_{1\leq i\leq M+1}\\\{b_i\}_{1\leq i\leq M}
		\end{array};x\right),\\
		\notag&\frac{(Ax/B)_\infty}{(x)_\infty}\sum_{l\in\mathbb{Z}_{\geq0}^M}\left(\frac{Ax}{B}\right)^{|l|}\frac{\Delta(\{b_iq^{l_i}\}_{1\leq i\leq M})}{\Delta(\{b_i\}_{1\leq i\leq M})}\prod_{i=1}^M \frac{\prod_{j=1}^{M+1}(b_i/a_j)_{l_i}}{\prod_{j=1}^M(qb_i/b_j)_{l_i}\cdot (b_i)_{l_i}}\\
		&={}_{M+1}\varphi_M\left(\begin{array}{c}
			\{a_i\}_{1\leq i\leq M+1}\\\{b_i\}_{1\leq i\leq M}
		\end{array};x\right)\quad \left(\left|\frac{Ax}{B}\right|<1\right),\label{generalized1}\\
		\notag&\frac{(a_{M+1}x)_\infty}{(x)_\infty}\sum_{l\in\mathbb{Z}_{\geq0}^M}\left(-\frac{Ax}{a_{M+1}B}\right)^{|l|}\frac{\Delta(\{b_iq^{l_i}\}_{1\leq i\leq M})}{\Delta(\{b_i\}_{1\leq i\leq M})}\frac{(a_{M+1})_{|l|}}{(a_{M+1}x)_{|l|}}\prod_{i=1}^M\frac{\prod_{j=1}^M(b_i/a_j)_{l_i}\cdot b_i^{l_i}q^{\binom{l_i}{2}}}{\prod_{j=1}^M (qb_i/b_j)_{l_i}\cdot (b_i)_{l_i}}\\
		&={}_{M+1}\varphi_M\left(\begin{array}{c}
			\{a_i\}_{1\leq i\leq M+1}\\\{b_i\}_{1\leq i\leq M}
		\end{array};x\right),\label{generalized2}
	\end{align}
	where $A=a_1\cdots a_{M+1}$ and $B=b_1\cdots b_M$.
\end{cor}
\begin{rem}
	When $M=1$, the formulas \eqref{generalized1} and \eqref{generalized2} reduce to $q$-Euler transformation and Jackson's transformation formula \cite{GR}:
	\begin{align}
		{}_2\varphi_1\left(\begin{array}{c}
			a,b\\c
		\end{array};x\right)&=\frac{(abx/c)_\infty}{(x)_\infty}{}_2\varphi_1\left(\begin{array}{c}
			c/a,c/b\\c
		\end{array};\frac{abx}{c}\right)\\
		&=\frac{(bx)_\infty}{(x)_\infty}{}_2\varphi_2\left(\begin{array}{c}
			b,c/a\\c,bx
		\end{array};ax\right).
	\end{align}
\end{rem}
\section{Summary and discussions}\label{secsum}
We summarize this paper.

The main results of this paper are Theorem \ref{thmkajiint} and Theorem \ref{thmmultisystem}.
In Theorem \ref{thmkajiint}, we gave a transformation formula between Kajihara's $q$-hypergeometric series $W^{M,2}$ and the Jackson integral of Riemann--Papperitz type:
\begin{align}
	\int_{q/a_i}^{q/a_j}\prod_{k=1}^{M+3}\frac{(a_kt)_\infty}{(b_kt)_\infty}d_qt\ \ (a_1\cdots a_{M+3}=q^2 b_1\cdots b_{M+3}).
\end{align}
This formula is a generalization of Bailey's formula \eqref{inttoser}.
We also derived a $q$-difference system $q$-$RP^M$ associated with this integral (Theorem \ref{thmmultisystem}).
It follows that we obtained a $q$-difference system satisfied by Kajihara's $q$-hypergeometric series $W^{M,2}$ (see Corollary \ref{corKajiharaeq}).
The system $q$-$RP^M$ is a multivariable extension of the $q$-Riemann--Papperitz system (see Remark \ref{remqRPM1}, \ref{remqRPM2}).
We gave a basis of the space of solutions for $q$-$RP^M$ in Theorem \ref{thmindependence}.

In section \ref{secdegene}, we considered degenerations of the system $q$-$RP^M$.
The degeneration $q$-$RP^M_{\mathrm{degene}}$, defined in \eqref{hatem2dash}--\eqref{multidegeneeq7}, of the system includes the $q$-Appell--Lauricella system.
From this inclusion, we gave solutions for the $q$-Appell--Lauricella system in terms of degenerations of $W^{M,2}$ (Corollary \ref{newcorALsol}), and we showed  transformation formulas between these solutions and the $q$-Appell--Lauricella series $\varphi_D$ (Proposition \ref{newpropALtrans}).


There are many related open problems.


\begin{itemize}
%
	\item[(1)] 
	We gave a $q$-difference system $q$-$RP^M$ satisfied by Kajihara's $q$-hypergeometric series $W^{M,2}$.
	The way to get this system is due to a transformation formula of the Jackson integral of Riemann--Papperitz type:
	\begin{align}\label{futureint}
		\int_C \prod_{k=1}^{M+3}\frac{(a_kt)_\infty}{(b_kt)_\infty}d_qt.
	\end{align}
	It is important to get the system directly from the definition of $W^{M,2}$.
	If we can derive the system directly, it would be extended {to} more general $q$-/elliptic hypergeometric series.
	Note that an elliptic analog of Kajihara's $q$-hypergeometric series and its transformation formula are studied in \cite{KN}.
	\item[(2)]
	We gave integral solutions for the $q$-difference system $q$-$RP^M$.
	The connection problem for the system $q$-$RP^M$, i.e. to find linear relations among these integrals, is important.
	This includes the connection problem for the variant of $q$-hypergeometric equation of degree three \cite{HMST} as a special case.
	The problem for $q$-$RP^M$ is discussed in \cite{FNconn}.
	\item[(3)]
	The $q$-Appell--Lauricella series
	\begin{align}
		\varphi_D\left(\begin{array}{c}
			a;\{b_i\}_{1\leq i\leq M}\\
			c
		\end{array};\{x_i\}_{1\leq i\leq M}\right),
	\end{align}
	is characterized as the solution of the form $1+O(|x_1/x_2|+\cdots+|x_{M-1}/x_M|+|x_M|)$ for the $q$-Appell--Lauricella system \cite{N}.
	In the general theory of $q$-difference systems in several variables, local solutions should be characterized by the asymptotic behavior near the singularity, i.e. origin or infinity, in some prescribed sector.
	For the general theory, see \cite{Ao1,Ao2}.
	The series solution $W$ \eqref{defW} of the multivariable system $q$-$RP^M$ should be characterized by some conditions.
	It is an important problem and would be related to the above problem (1).
	In the last of section \ref{secdegene}, we get relations among the $q$-Appell--Lauricella series $\varphi_D$ and degenerations of $W^{M,2}$ by using a characterization of solutions for the $q$-Appell--Lauricella system.
	Hence, a characterization for the solution $W^{M,2}$  may be also useful to get some transformation formulas.
	\item[(4)]
	In this paper, we studied the Jackson integral \eqref{futureint} of Riemann--Papperitz type.
	There are many interesting hypergeometric type Jackson integral.
	For example, a $q$-Selberg type integral
	\begin{align}
		\int_C \prod_{i=1}^n \left(t_i^\alpha \prod_{r=1}^m\frac{(q a_r^{-1}t_i)_\infty}{(b_r t_i)_\infty}\right)\prod_{1\leq j<k\leq n}t_j^{2\tau-1}\frac{(q^{1-\tau}t_k/t_j)_\infty}{(q^\tau t_k/t_j)_\infty}\frac{d_qt_1}{t_1}\wedge\cdots \wedge\frac{d_qt_n}{t_n},
	\end{align}
	is an active area of research for special functions, combinatorics, mathematical physics and orthogonal polynomials.
	For details and histories, see \cite{It,It2023}.
	To consider a ``Riemann--Papperitz type extension of the $q$-Selberg integral'' is interesting and important.
	In addition, there are many important integrals for the differential case.
	{Examples include the Aomoto--Gelfand hypergeometric integrals and the GKZ hypergeometric integrals (see \cite{AK,GKZint,Matsubara,Y}).}
	To study $q$-analogs of these integrals is also interesting.
	We remark that a system of first order $q$-difference equations, i.e. a Pfaffian system, for the $q$-Selberg integral is discussed in \cite{It,It2023}.
	Rewriting $q$-$RP^M$ to a Pfaffian system is valuable.
	%
	\item[(5)]
	In \cite{NY1,NY2}, the $q$-Garnier system, which is a multivariable extension of the $q$-Painlev\'{e} VI equation, is considered by means of the Pad\'{e} interpolation of some product $\psi$ of the $q$-shifted factorial.
	For the Pad\'{e} method to Painlev\'{e} equations, see also a recent book \cite{NYbook}.
	In \cite{SZ}, the Pad\'{e} interpolation of some product of Jacobi's theta function is expressed explicitly in terms of the very well-poised elliptic  hypergeometric series ${}_{12}E_{11}$.
	This function ${}_{12}E_{11}$ has a transformation formula of the form:
	\begin{align}
		{}_{12}E_{11}=(\mbox{product of theta functions})\times {}_{12}E_{11}.
	\end{align}
	That is the reason why the Pad\'{e} problem is solved by using ${}_{12}E_{11}$.
	In this paper we got some transformation formulas for Kajihara's $q$-hypergeometric function $W^{M,2}$ (see Corollary \ref{corkajilinear}).
	{It would be expected that the Pad\'{e} problem or its certain generalization associated with $\psi$ can be solved by using $W^{M,2}$ similar to Spiridonov--Zhedanov's method \cite{SZ}, and a special solution for the $q$-Garnier system in terms of $W^{M,2}$ may be obtained.}
	
	
	\item[(6)]
	 The function ${}_8W_7$ is also known as the Askey--Wilson function \cite{KS}, related to orthogonal polynomials.
	The Koornwinder polynomial is a multivariable extension of the Askey--Wilson polynomial, and it is an eigenfunction of Koornwinder's $q$-difference operator.
	To give a solution for the bispectral problem of Koornwinder operator explicitly is important for mathematical physics (see \cite{vM,vMS,S}, and \cite{NS} for an explicit form of Macdonald's case).
	In this paper we discussed a certain extension $W^{M,2}$ of the Askey--Wilson function from the viewpoint of the integral representation of it.
	We hope that the series $W^{M,2}$ is an eigenfunction of Koornwinder operator of special eigenvalue.
	Note that an integral solution for the bispectral problem of the Macdonald's operator is given by a multiple Jackson integral in \cite{NS}.
	For general eigenvalue, a multiple extension of \eqref{intmulti} would be needed.
	
	\item[(7)]
	The variants of $q$-hypergeometric equation \cite{HMST} are introduced as special cases of degenerations for the eigenvalue problem of Ruijsenaars--van Diejen operator of one variable case.
	The system $q$-$RP^M$ is a multivariable extension of this variant.
	Its relation to the Ruijsenaars--van Diejen operator \cite{vD,Ru} is an interesting problem.
	
\end{itemize}
%

\appendix
\section{Appendix}\label{appA}
In this appendix, we check convergences of multiple series and justify termwise limits studied in this paper.
We first remark that discussion in \cite[Appendix B]{MN2012} is useful and informative.

The following lemmas are fundamental.
\begin{lemm}[multiple ratio test, cf \cite{AppKam,MN2012}]
	We put
	\begin{align}
		g_m(l_1,\ldots,l_M)=\left|\frac{f(l_1,\ldots,l_{m-1},l_m+1,l_{m+1},\ldots,l_M)}{f(l_1,\ldots,l_M)}\right|\quad (1\leq m\leq M).
	\end{align}
	If 
	\begin{align}
		\lim_{\epsilon\to\infty}g_m(\epsilon l_1,\ldots,\epsilon l_M)<1\quad (1\leq m\leq M),
	\end{align}
	then the multiple sum
	\begin{align}
		\sum_{l\in\mathbb{Z}_{\geq0}^M}f(l_1,\ldots,l_M),
	\end{align}
	converges absolutely.
\end{lemm}
\begin{lemm}[Tannery's theorem, cf \cite{Br}]
	If $|p_l(k)|<P_l$ and $\displaystyle \sum_{l\in\mathbb{Z}_{\geq0}^M}P_l<\infty$, then we have
	\begin{align}
		\lim_{n\to\infty}\sum_{l\in\mathbb{Z}_{\geq0}^M}p_l(k)=\sum_{l\in\mathbb{Z}_{\geq0}}\tilde{p}_l,
	\end{align}
	where $\tilde{p}_l=\lim_{k\to\infty}p_l(k)$.
\end{lemm}

{
Due to the limit
\begin{align}
	&\lim_{k\to\infty}(aq^k)_l=1,
\end{align}
there exist positive values $D_1$ and $D_2$ which are independent from $k$ and $l$, such that
\begin{align}
	D_1<|(aq^k)_l|<D_2.\label{ine1}
\end{align}
We also have
\begin{align}
	D_3<|(aq^{k-l})_l|<D_4,\label{ine2}
\end{align}
where $k-l\geq0$, $D_3$ and $D_4$ are positive values independent from $k$ and $l$.
By the formula
\begin{align}
	(aq^{-k})_l=(-a)^l q^{\binom{l}{2}} (q^{-l+1+k}/a)_l,
\end{align}
there exist positive values $D_5$ and $D_6$ which are independent from $k$ and $l$, such that
\begin{align}
	D_5|a^l q^{\binom{l}{2}}|<|(aq^{-k})_l|<D_6|a^lq^{\binom{l}{2}}|.\label{ine2.5}
\end{align}
We also have
\begin{align}
	D_7\left|\frac{b}{a}\right|^l<\left|\frac{(aq^{-k})_l}{(bq^{-k})_l}\right|<D_8\left|\frac{b}{a}\right|^l,\label{ine3}
\end{align}
where $k-l\geq0$, $D_7$ and $D_8$ are positive values independent from $k$ and $l$.
The above inequalities are used frequently in the following.
}

\subsection{Convergence of Kajihara's series \eqref{defkaji}}\label{appA1}
	We put
	\begin{align}
		&f(l_1,\ldots,l_M)=z^{|l|}\frac{\Delta(x q^l)}{\Delta(x)}\prod_{i=1}^M\left(\frac{1-ax_iq^{|l|+l_i}}{1-ax_i}\frac{(ax_i)_{|l|}\prod_{j=1}^{M+N}(x_iu_j)_{l_i}}{\prod_{j=1}^M(qx_i/x_j)_{l_i}\prod_{k=1}^N(aqx_i/v_k)_{l_i}}\right)\frac{\prod_{j=1}^N(v_k)_{|l|}}{\prod_{j=1}^{M+N}(aq/u_j)_{|l|}},\\
		&g_m(l_1,\ldots,l_M)=\left|\frac{f(l_1,\ldots,l_{m-1},l_m+1,l_{m+1},\ldots,l_M)}{f(l_1,\ldots,l_M)}\right|\quad (1\leq m\leq M).
	\end{align}
	Then the series \eqref{defkaji} is given as
	\begin{align}
		\sum_{l\in\mathbb{Z}_{\geq0}^M}f(l_1,\ldots,l_M).
	\end{align}
	We easily find 
	\begin{align}
		\lim_{\epsilon\to\infty}g_m(\epsilon l_1,\ldots,\epsilon l_M)=|z|.
	\end{align}
	Due to the multiple ratio test, Kajihara's series \eqref{defkaji} converges if $|z|<1$.
	
\subsection{Justification of the limit \eqref{just1}: $k\to\infty$}
We put
\begin{align}
	&P_l=q^{|l|}\frac{\Delta(xq^l)}{\Delta(x)}\prod_{i=1}^M\left(\frac{1-ax_iq^{|l|+l_i}}{1-ax_i}\frac{(ax_i)_{|l|}\prod_{j=1}^{M+2}(x_ib_j)_{l_i}}{\prod_{j=1}^M (qx_i/x_j)_{l_i}\cdot(aqx_i/c,aqx_i/(\mu c))_{l_i}}\right)\frac{(c,\mu c)_{|l|}}{\prod_{j=1}^{M+2}(aq/b_j)_{|l|}},
	\\
	&p_l(k)=P_l\cdot \prod_{i=1}^M \frac{(x_ib_{M+3} q^k)_{l_i}}{(aqx_iq^k)_{l_i}}\cdot \frac{(q^{-k})_{|l|}}{(aqq^{-k}/b_{M+3})_{|l|}}.
\end{align}
Due to \eqref{ine1} and \eqref{ine3},there exist a positive constant $D$ which is independent from $k$ and $l$ such that
\begin{align}
	|p_l(k)|<D|P_l|\left|\frac{b_{M+3}}{aq}\right|^l.
\end{align}
The series
\begin{align}
	\sum_{l\in\mathbb{Z}_{\geq0}^M}|P_l|\left|\frac{b_{M+3}}{aq}\right|^l,
\end{align}
converges if $|b_{M+3}/a|<1$.
This convergence can be checked similar to discussion in Appendix \ref{appA1}.
By Tannery's theorem, we have
\begin{align}
	\lim_{k\to\infty}\sum_{l\in\mathbb{Z}_{\geq0}^M}p_l(k)=\sum_{l\in\mathbb{Z}_{\geq0}^M}\lim_{k\to\infty}p_l(k).
\end{align}
This justify the limit \eqref{just1}.

\subsection{Justification of the limits \eqref{just2} and \eqref{just3}: $k\to\infty$}
We recall that $k=2K-1$ ,$\mu=a^3q^2/(c^2b_1b_2\cdots b_{M+2}x_1x_2\cdots x_M)$ and
\begin{align}
	C_l=\frac{1-\mu q^{-k}q^{2l}}{1-\mu q^{-k}}\prod_{i=1}^M\frac{(\mu cq^{-k}/(ax_i))_l}{(aqx_i/c)_l}\prod_{j=1}^{M+2}\frac{(aq/(cb_j))_l}{(\mu cb_jq^{-k}/a)_l}\cdot\frac{(\mu q^{-k},a q q^{-k}/(c b_{M+3}),\mu c,q^{-k})_l}{(q,\mu c b_{M+3}/a,q q^{-k}/c,q \mu)_l}q^l.
\end{align}
Due to \eqref{ine3}, for $0\leq l\leq K-1$, there is a positive constant $D$ independent from $k$ and $l$ such that
\begin{align}
	|C_l|<D\left|\prod_{i=1}^M\frac{1}{(aqx_i/c)_l}\prod_{j=1}^{M+2}{(aq/(cb_j))_l}\cdot\frac{(\mu c)_l}{(q,\mu c b_{M+3}/a,q \mu)_l}q^l\right|.
\end{align}
The series
\begin{align}
	\sum_{l=0}^\infty\left|\prod_{i=1}^M\frac{1}{(aqx_i/c)_l}\prod_{j=1}^{M+2}{(aq/(cb_j))_l}\cdot\frac{(\mu c)_l}{(q,\mu c b_{M+3}/a,q \mu)_l}q^l\right|,
\end{align}
converges, hence, the limit \eqref{just2} is justified by Tannery's theorem.
The limit \eqref{just3} is also justified in a similar way.
Key formulas for this case are $(a)_{k-l}=(a)_k/(aq^{k-l})_l$ and $(aq^{-l})_l=(-a)^{l}q^{-l(l-1)/2}(q/a)_l$.

\subsection{Justification of the limit \eqref{just4}: $a_{M+3}\to\infty$}
We recall that $b_{M+3}=q^\lambda a_{M+3}$, $|q^{\lambda+1}|<1$ and
\begin{align}
	C_0(t)=t^\lambda\frac{\theta(b_{M+3}t)}{\theta(a_{M+3}t)}.
\end{align}
For $1\leq j\leq M+2$, we have
\begin{align}
	\notag &C_0\left(\frac{q}{a_j}\right)\int_0^{q/a_j}\prod_{i=1}^{M+3}\frac{(a_i t)_\infty}{(b_i t)_\infty}d_qt\\
	\notag&=\int_0^{q/a_j}C_0(t)\prod_{i=1}^{M+3}\frac{(a_i t)_\infty}{(b_i t)_\infty}d_qt\\
	&=(1-q)\left(\frac{q}{a_j}\right)^{\lambda +1}\prod_{i=1}^{M+2}\frac{(a_iq/a_j)_\infty}{(b_iq/a_j)_\infty}\cdot \frac{(a_j/b_{M+3})_\infty}{(a_j/a_{M+3})_\infty}\sum_{l=0}^\infty\prod_{i=1}^{M+2}\frac{(b_iq/a_j)_l}{(a_iq/a_j)_l}\cdot \frac{(a_jq^{-l}/a_{M+3})_l}{(a_jq^{-l}/(q^\lambda a_{M+3}))_l}q^l.
\end{align}
According to \eqref{ine3}, there exists a positive constant $D$ independent from $a_{M+3}$ and $l$ such that
\begin{align}
	\left|\frac{(a_jq^{-l}/a_{M+3})_l}{(a_jq^{-l}/(q^\lambda a_{M+3}))_l}\right|<D|q^\lambda|.
\end{align}
Since $|q^{\lambda +1}|<1$, the series 
\begin{align}
	\sum_{l=0}^\infty\left|\prod_{i=1}^{M+2}\frac{(b_iq/a_j)_l}{(a_iq/a_j)_l}\cdot {(q^{\lambda+1})}^l\right|
\end{align}
converges.
Therefore the limit \eqref{just4} are justified by Tannery's theorem.

\subsection{Justification of the limit \eqref{just5}: $a_{M+2}\to0$}
We recall $b_{M+2}=q^\beta a_{M+2}$.
For $1\leq j\leq M+1$, we have
\begin{align}
	\notag&\int_0^{q/a_j}\prod_{i=1}^{M+2}t^\lambda\frac{(a_i t)_\infty}{(b_i t)_\infty}d_qt\\
	&=(1-q)\left(\frac{q}{a_j}\right)^{\lambda+1}\prod_{i=1}^{M+2}\frac{(a_i q/a_j)_\infty}{(b_i q/a_j)_\infty}\sum_{l=0}^\infty\prod_{i=1}^{M+1}\frac{(b_iq/a_j)_l}{(a_iq/a_j)_l}\cdot \frac{(a_{M+2}q/a_j)_l}{(q^\beta a_{M+2}q/a_j)_l}(q^{\lambda+1})^l.
\end{align}
Using \eqref{ine1}, we have
\begin{align}
	\left|\frac{(a_{M+2}q/a_j)_l}{(b_{M+2}q/a_j)_l}\right|<D,
\end{align}
where $D$ is a constant independent from $a_{M+2}$ and $l$
When $|q^{\lambda+1}|<1$, the series 
\begin{align}
	\sum_{l=0}^\infty\left|\prod_{i=1}^{M+1}\frac{(b_iq/a_j)_l}{(a_iq/a_j)_l}(q^{\lambda+1})^l\right|,
\end{align}
converges, therefore the limit \eqref{just5} is justified.

\subsection{Justification of the limits \eqref{just6} and \eqref{just7}: $a_{M+3}\to\infty$}
We recall that
\begin{align}
	\notag&W^{M,2}\left(\{a_i\}_{1\leq i\leq M};\frac{qb_{M+3}}{a_{M+2}a_{M+3}};\left\{\frac{1}{b_j}\right\}_{1\leq j\leq M+2};\frac{qb_{M+3}}{a_{M+2}},\frac{qb_{M+3}}{a_{M+3}};\frac{a_{M+1}}{b_{M+3}}\right)\\
	\notag&=\sum_{l\in\mathbb{Z}_{\geq0}^M}\left(\frac{a_{M+1}}{b_{M+3}}\right)^{|l|}\frac{\Delta (\{a_iq^{l_i}\}_{1\leq i\leq M})}{\Delta(\{a_i\}_{1\leq i\leq M})}\frac{(qb_{M+3}/a_{M+2},qb_{M+3}/a_{M+3})_{|l|}}{\prod_{j=1}^{M+2}(q^2 b_jb_{M+3}/(a_{M+2}a_{M+3}))_{|l|}}\\
	&\phantom{=}\times\prod_{i=1}^M\left(\frac{1-(qb_{M+3}/(a_{M+2}a_{M+3}))a_iq^{|l|+l_i}}{1-(qb_{M+3}/(a_{M+2}a_{M+3}))a_i}\frac{(qa_ib_{M+3}/(a_{M+2}a_{M+3}))_{|l|}\prod_{j=1}^{M+2}(a_i/b_j)_{l_i}}{\prod_{j=1}^M(qa_i/a_j)_{l_i} \cdot(qa_i/a_{M+2},qa_i/a_{M+3})_{l_i}}\right).
\end{align}
We put
\begin{align}
	\notag&p_l=\left(\frac{a_{M+1}}{b_{M+3}}\right)^{|l|}\frac{\Delta (\{a_iq^{l_i}\}_{1\leq i\leq M})}{\Delta(\{a_i\}_{1\leq i\leq M})}\frac{(qb_{M+3}/a_{M+2},qb_{M+3}/a_{M+3})_{|l|}}{\prod_{j=1}^{M+2}(q^2 b_jb_{M+3}/(a_{M+2}a_{M+3}))_{|l|}}\\
	&\phantom{=}\times\prod_{i=1}^M\left(\frac{1-(qb_{M+3}/(a_{M+2}a_{M+3}))a_iq^{|l|+l_i}}{1-(qb_{M+3}/(a_{M+2}a_{M+3}))a_i}\frac{(qa_ib_{M+3}/(a_{M+2}a_{M+3}))_{|l|}\prod_{j=1}^{M+2}(a_i/b_j)_{l_i}}{\prod_{j=1}^M(qa_i/a_j)_{l_i} \cdot(qa_i/a_{M+2},qa_i/a_{M+3})_{l_i}}\right).
\end{align}
First we justify the limit \eqref{just6}.
We have
\begin{align}
	\notag&p_l |_{b_{M+3}=q^\lambda a_{M+3}}=a_{M+1}^{|l|}\frac{\Delta (\{a_iq^{l_i}\}_{1\leq i\leq M})}{\Delta(\{a_i\}_{1\leq i\leq M})}\frac{(q^{\lambda+1}a_{M+3}/a_{M+2},q^{\lambda+1})_{|l|}}{\prod_{j=1}^{M+2}(q^{\lambda+2} b_j/a_{M+2})_{|l|}}\\
	\notag&\phantom{=}\times\prod_{i=1}^M\left(\frac{1-(q^{\lambda+1}/a_{M+2})a_iq^{|l|+l_i}}{1-(q^{\lambda+1}/a_{M+2})a_i}\frac{(q^{\lambda+1}a_i/a_{M+2})_{|l|}\prod_{j=1}^{M+2}(a_i/b_j)_{l_i}}{\prod_{j=1}^M(qa_i/a_j)_{l_i} \cdot(qa_i/a_{M+2})_{l_i}}\right)\\
	&\phantom{=}\times \left(\frac{1}{q^\lambda a_{M+3}}\right)^{|l|}(q^{\lambda+1}a_{M+3}/a_{M+2})_{|l|}\prod_{i=1}^M\frac{1}{(qa_i/a_{M+3})_{l_i}}.
\end{align}
Due to \eqref{ine1} and \eqref{ine2.5}, there exists a positive constant $D$ independent from $a_{M+3}$ and $l$ such that
\begin{align}
	\left|\left(\frac{1}{q^\lambda a_{M+3}}\right)^{|l|}(q^{\lambda+1}a_{M+3}/a_{M+2})_{|l|}\prod_{i=1}^M\frac{1}{(qa_i/a_{M+3})_{l_i}}\right|<D\left|\left(\frac{q}{a_{M+2}}\right)^{|l|}q^{\binom{|l|}{2}}\right|.
\end{align}
Using the multiple ratio test, we find that the series
\begin{align}
	\notag&\sum_{l\in\mathbb{Z}_{\geq0}}\biggl|a_{M+1}^{|l|}\frac{\Delta (\{a_iq^{l_i}\}_{1\leq i\leq M})}{\Delta(\{a_i\}_{1\leq i\leq M})}\frac{(q^{\lambda+1}a_{M+3}/a_{M+2},q^{\lambda+1})_{|l|}}{\prod_{j=1}^{M+2}(q^{\lambda+2} b_j/a_{M+2})_{|l|}}\\
	\notag&\phantom{===}\times\prod_{i=1}^M\left(\frac{1-(q^{\lambda+1}/a_{M+2})a_iq^{|l|+l_i}}{1-(q^{\lambda+1}/a_{M+2})a_i}\frac{(q^{\lambda+1}a_i/a_{M+2})_{|l|}\prod_{j=1}^{M+2}(a_i/b_j)_{l_i}}{\prod_{j=1}^M(qa_i/a_j)_{l_i} \cdot(qa_i/a_{M+2})_{l_i}}\right)\biggr|\\
	&\phantom{===}\times\left|\left(\frac{q}{a_{M+2}}\right)^{|l|}q^{\binom{|l|}{2}}\right|,
\end{align}
converges for any $a_i$, $b_i$ and $\lambda$.
The key factor is $q^{\binom{|l|}{2}}$.
This factor satisfies
\begin{align}
	\left|\frac{q^{\binom{\epsilon l_1+\cdots+\epsilon l_{m-1}+\epsilon l_{m}+1+\epsilon l_{m+1}+\cdots+\epsilon l_M}{2}}}{q^{\binom{\epsilon l_1+\cdots+\epsilon l_M}{2}}}\right|=|q^{\epsilon l_1+\cdots+\epsilon l_M}|\xrightarrow{\epsilon\to\infty}0.
\end{align}
Therefore the limit \eqref{just6} is justified by Tannery's theorem.

Next, we justify the limit \eqref{just7}.
We have
\begin{align}
	\notag&(s(b_{M+2},b_{M+3}).p_l)|_{b_{M+3}=q^\lambda a_{M+3}}\\
	\notag&=\left(\frac{a_{M+1}}{b_{M+2}}\right)^{|l|}\frac{\Delta (\{a_iq^{l_i}\}_{1\leq i\leq M})}{\Delta(\{a_i\}_{1\leq i\leq M})}\frac{(qb_{M+2}/a_{M+2})_{|l|}}{ (q^{\lambda+2}b_{M+2}/a_{M+2})_{|l|}}\prod_{i=1}^M\left(\frac{\prod_{j=1}^{M+1}(a_i/b_j)_{l_i}}{\prod_{j=1}^M(qa_i/a_j)_{l_i} \cdot(qa_i/a_{M+2})_{l_i}}\right).
	\\
	\notag&\phantom{=}\times \prod_{i=1}^M\left(\frac{1-(qb_{M+2}/(a_{M+2}a_{M+3}))a_iq^{|l|+l_i}}{1-(qb_{M+2}/(a_{M+2}a_{M+3}))a_i}\frac{(qa_ib_{M+2}/(a_{M+2}a_{M+3}))_{|l|} (a_iq^{-\lambda}/a_{M+3})_{l_i}}{(qa_i/a_{M+3})_{l_i}}\right)\\
	&\phantom{=}\times \frac{(qb_{M+2}/a_{M+3})_{|l|}}{\prod_{j=1}^{M+1}(q^2 b_jb_{M+2}/(a_{M+2}a_{M+3}))_{|l|}}.
\end{align}
Using \eqref{ine1}, we have
\begin{align}
	\notag&\phantom{=}\biggl| \prod_{i=1}^M\left(\frac{1-(qb_{M+2}/(a_{M+2}a_{M+3}))a_iq^{|l|+l_i}}{1-(qb_{M+2}/(a_{M+2}a_{M+3}))a_i}\frac{(qa_ib_{M+2}/(a_{M+2}a_{M+3}))_{|l|} (a_iq^{-\lambda}/a_{M+3})_{l_i}}{(qa_i/a_{M+3})_{l_i}}\right)\\
	&\phantom{=}\times \frac{(qb_{M+2}/a_{M+3})_{|l|}}{\prod_{j=1}^{M+1}(q^2 b_jb_{M+2}/(a_{M+2}a_{M+3}))_{|l|}}\biggr|<D,
\end{align}
where $D$ is a positive constant independent from $a_{M+3}$ and $l$.
According to the multiple ratio test, the series
\begin{align}
	\sum_{l\in\mathbb{Z}_{\geq0}^M}\left|\left(\frac{a_{M+1}}{b_{M+2}}\right)^{|l|}\frac{\Delta (\{a_iq^{l_i}\}_{1\leq i\leq M})}{\Delta(\{a_i\}_{1\leq i\leq M})}\frac{(qb_{M+2}/a_{M+2})_{|l|}}{ (q^{\lambda+2}b_{M+2}/a_{M+2})_{|l|}}\prod_{i=1}^M\left(\frac{\prod_{j=1}^{M+1}(a_i/b_j)_{l_i}}{\prod_{j=1}^M(qa_i/a_j)_{l_i} \cdot(qa_i/a_{M+2})_{l_i}}\right)\right|,
\end{align}
converges if $|a_{M+1}/b_{M+2}|<1$.
Therefore the limit \eqref{just7} is justified by Tannery's theorem.

\subsection{Justification of the limit \eqref{just8}: $a_{M+2}\to0$}
We put
\begin{align}
	\notag&p_l=\left(-\frac{q}{a_{M+1}}\right)^{|l|}q^{\binom{|l|}{2}}\frac{\Delta (\{a_iq^{l_i}\}_{1\leq i\leq M})}{\Delta(\{a_i\}_{1\leq i\leq M})}\frac{(q^{\lambda+1})_{|l|}}{\prod_{j=1}^{M+2}( q^{\lambda+2} b_j/a_{M+1})_{|l|}}\\
	\notag&\phantom{=}\times\prod_{i=1}^M\left(\frac{1-(q^{\lambda+1} /a_{M+1})a_iq^{|l|+l_i}}{1-(q^{\lambda+1} /a_{M+1})a_i}\frac{(q^{\lambda+1} a_i/a_{M+1})_{|l|}\prod_{j=1}^{M+1}(a_i/b_j)_{l_i}}{\prod_{j=1}^{M+1}(qa_i/a_j)_{l_i} }\right)\\
	&\phantom{=}\times a_{M+2}^{|l|} \prod_{i=1}^M(a_iq^{-\beta}/a_{M+2})_{l_i}.
\end{align}
Due to \eqref{ine2.5}, there exists a positive constant $D$ independent from $a_{M+2}$ and $l$ such that
\begin{align}
	\left|a_{M+2}^{|l|} \prod_{i=1}^M(a_iq^{-\beta}/a_{M+2})_{l_i}\right|<D \left|\prod_{i=1}^M \left((a_i/q^{\beta})^{l_i}q^{\binom{l_i}{2}}\right)\right|.
\end{align}
By the multiple ratio test, we find that the series
\begin{align}
	\notag&\sum_{l\in\mathbb{Z}_{\geq0}^M}\biggl|\left(-\frac{q}{a_{M+1}}\right)^{|l|}q^{\binom{|l|}{2}}\frac{\Delta (\{a_iq^{l_i}\}_{1\leq i\leq M})}{\Delta(\{a_i\}_{1\leq i\leq M})}\frac{(q^{\lambda+1})_{|l|}}{\prod_{j=1}^{M+2}( q^{\lambda+2} b_j/a_{M+1})_{|l|}}\\
	\notag&\phantom{=}\times\prod_{i=1}^M\left(\frac{1-(q^{\lambda+1} /a_{M+1})a_iq^{|l|+l_i}}{1-(q^{\lambda+1} /a_{M+1})a_i}\frac{(q^{\lambda+1} a_i/a_{M+1})_{|l|}\prod_{j=1}^{M+1}(a_i/b_j)_{l_i}}{\prod_{j=1}^{M+1}(qa_i/a_j)_{l_i} }\right)\biggr|\\
	&\phantom{=}\times \left|\prod_{i=1}^M \left((a_i/q^{\beta})^{l_i}q^{\binom{l_i}{2}}\right)\right|,
\end{align}
converges for any $a_i$, $b_i$, $\lambda$ and $\beta$.
More precisely, the factors $q^{\binom{|l|}{2}}$ and $q^{\binom{l_m}{2}}$ satisfy
\begin{align}
	&\left|\frac{q^{\binom{\epsilon l_1+\cdots+\epsilon l_{m-1}+\epsilon l_{m}+1+\epsilon l_{m+1}+\cdots+\epsilon l_M}{2}}}{q^{\binom{\epsilon l_1+\cdots+\epsilon l_M}{2}}}\right|=|q^{\epsilon l_1+\cdots+\epsilon l_M}|\xrightarrow{\epsilon\to\infty}0,\\
	&\left|\frac{q^{\binom{\epsilon l_m+1}{2}}}{q^{\binom{\epsilon l_m}{2}}}\right|=|q^{\epsilon l_m}|\xrightarrow{\epsilon\to\infty}0.
\end{align}
Then the above series converges for any $a_i$, $b_i$, $\lambda$ and $\beta$.
Therefore the limit \eqref{just8} is justified by Tannery's theorem.

\subsection{Justification of the limits \eqref{just9} and \eqref{just10}: $a_{M+2}\to0$}

We put
\begin{align}
	p_l=\left(\frac{a_{M+1}}{b_{M+2}}\right)^{|l|}\frac{\Delta (\{a_iq^{l_i}\}_{1\leq i\leq M})}{\Delta(\{a_i\}_{1\leq i\leq M})}\frac{(qb_{M+2}/a_{M+2})_{|l|}}{(q^{\lambda+2} b_{M+2}/a_{M+2})_{|l|}}\prod_{i=1}^M\left(\frac{\prod_{j=1}^{M+1}(a_i/b_j)_{l_i}}{\prod_{j=1}^M(qa_i/a_j)_{l_i} \cdot(qa_i/a_{M+2})_{l_i}}\right).
\end{align}
First we check the limit \eqref{just9}.
We have
\begin{align}
	\notag&(s(a_{M+1},a_{M+2}).p_l)|_{b_{M+2}=q^\beta a_{M+2}}\\
	&=\left(q^{-\beta}\right)^{|l|}\frac{\Delta (\{a_iq^{l_i}\}_{1\leq i\leq M})}{\Delta(\{a_i\}_{1\leq i\leq M})}\prod_{i=1}^M\left(\frac{\prod_{j=1}^{M+1}(a_i/b_j)_{l_i}}{\prod_{j=1}^{M+1}(qa_i/a_j)_{l_i}}\right)\times\frac{(q^{\beta+1}a_{M+2}/a_{M+1})_{|l|}}{(q^{\lambda+\beta+2} a_{M+2}/a_{M+1})_{|l|}}.
\end{align}
Due to \eqref{ine1}, we have 
\begin{align}
	\left|\frac{(q^{\beta+1}a_{M+2}/a_{M+1})_{|l|}}{(q^{\lambda+\beta+2} a_{M+2}/a_{M+1})_{|l|}}\right|<D,
\end{align}
where $D$ is a positive constant independent from $a_{M+2}$ and $l$.
According to the multiple ratio test, the series
\begin{align}
	\sum_{l\in\mathbb{Z}_{\geq0}^M}\left|\left(q^{-\beta}\right)^{|l|}\frac{\Delta (\{a_iq^{l_i}\}_{1\leq i\leq M})}{\Delta(\{a_i\}_{1\leq i\leq M})}\prod_{i=1}^M\left(\frac{\prod_{j=1}^{M+1}(a_i/b_j)_{l_i}}{\prod_{j=1}^{M+1}(qa_i/a_j)_{l_i}}\right)\right|,
\end{align}
converges if $|q^{-\beta}|<1$.
Therefore the limit \eqref{just9} is justified by Tannery's theorem.

Next, we check the limit \eqref{just10}.
We have
\begin{align}
	\notag&(s(a_{M+1},a_{M+2}).s(b_{M+1},b_{M+2}).p_l)|_{b_{M+2}=q^\beta a_{M+2}}\\
	&=\left(\frac{1}{b_{M+1}}\right)^{|l|}\frac{\Delta (\{a_iq^{l_i}\}_{1\leq i\leq M})}{\Delta(\{a_i\}_{1\leq i\leq M})}\frac{(qb_{M+1}/a_{M+1})_{|l|}}{(q^{\lambda+2} b_{M+1}/a_{M+1})_{|l|}}\prod_{i=1}^M\left(\frac{\prod_{j=1}^{M}(a_i/b_j)_{l_i}}{\prod_{j=1}^{M+1}(qa_i/a_j)_{l_i}}\right)\times a_{M+2}^{|l|}\prod_{i=1}^M (a_i q^{-\beta}/a_{M+2})_{l_i}.
\end{align}
Due to \eqref{ine2.5}, there exists a positive constant $D$ independent from $a_{M+2}$ and $l$ such that
\begin{align}
	\left|a_{M+2}^{|l|}\prod_{i=1}^M (a_i q^{-\beta}/a_{M+2})_{l_i}\right|<D \left|\prod_{i=1}^M\left(\frac{a_i}{q^\beta}\right)^{l_i}q^{\binom{l_i}{2}}\right|.
\end{align}
Thanks to the factor $q^{\binom{l_i}{2}}$, the series
\begin{align}
	\sum_{l\in\mathbb{Z}_{\geq0}^M}\left|\left(\frac{1}{b_{M+1}}\right)^{|l|}\frac{\Delta (\{a_iq^{l_i}\}_{1\leq i\leq M})}{\Delta(\{a_i\}_{1\leq i\leq M})}\frac{(qb_{M+1}/a_{M+1})_{|l|}}{(q^{\lambda+2} b_{M+1}/a_{M+1})_{|l|}}\prod_{i=1}^M\left(\frac{\prod_{j=1}^{M}(a_i/b_j)_{l_i}}{\prod_{j=1}^{M+1}(qa_i/a_j)_{l_i}}\right)\times\prod_{i=1}^M\left(\frac{a_i}{q^\beta}\right)^{l_i}q^{\binom{l_i}{2}} \right|,
\end{align}
converges for any $a_i$, $b_i$, $\lambda$ and $\beta$.
Therefore the limit \eqref{just10} is justified by Tannery's theorem.

\section*{Acknowledgement}
The author would like to thank Yasuhiko Yamada for valuable suggestions, helpful discussions and constant encouragements.
He also thank the referees for careful reading the manuscript and providing useful references and comments.
This work is supported by JST SPRING, Grant Number JPMJSP2148 and JSPS KAKENHI Grant Number 22H01116.

\end{document}